\documentclass[12pt,oneside,notitlepage]{book}
\usepackage[leqno]{amsmath}
\usepackage{ amsmath, amsthm, amsfonts,amssymb, hyperref, commath}
\usepackage{hyperref, graphicx, ulem, mathrsfs, ifpdf, multicol, color, verbatim, epsfig, xy}
\usepackage[toc,page]{appendix}
\xyoption{all}
\hypersetup{
   bookmarks=true,         % show bookmarks bar?
   unicode=false,          % non-Latin characters in Acrobat bookmarks
   pdftoolbar=true,        % show Acrobat toolbar?
   pdfmenubar=true,        % show Acrobat menu?
   pdffitwindow=false,     % window fit to page when opened
   pdfstartview={FitH},    % fits the width of the page to the window
   pdftitle={},    % title
   pdfauthor={},     % author
   pdfsubject={},   % subject of the document
   pdfcreator={},   % creator of the document
   pdfproducer={}, % producer of the document
   pdfkeywords={}, % list of keywords
   pdfnewwindow=true,      % links in new window
   colorlinks=true,       % false: boxed links; true: colored links
   linkcolor=blue,          % color of internal links
   citecolor=blue,        % color of links to bibliography
   filecolor=magenta,      % color of file links
   urlcolor=MidnightBlue          % color of external links
}
\usepackage[letterpaper, top=.8in, bottom=.8in, left=1.4in, right=.8in, includefoot]{geometry}

\begin{document}

\newcommand{\R}{{\mathbb R}}
\newcommand{\eqr}[1]{(\ref{#1})}
\newcommand{\h}{\mathbb H^3}
\newcommand{\C}{\mathbb C}
\newcommand{\s}{{\mathbb S}}
\newcommand{\n}{\noindent}
\pagestyle{plain} \pagenumbering{roman}

% The title page gets special treatment
% Begin the title page
\thispagestyle{empty}
\begin{center}
\vspace*{1in}
{\Large CLASSIFICATION AND ANALYSIS OF LOW INDEX MEAN CURVATURE FLOW SELF-SHRINKERS}\\
{\large by\\
Caleb Hussey}\\
\vfill A dissertation submitted to The Johns Hopkins University in
conformity with the requirements for the degree of Doctor of Philosophy\\
Baltimore, Maryland\\
June, 2012\\
\vfill
\copyright Caleb Hussey 2012\\
All rights reserved\\
%Updated 3-6-05
\end{center}
% End the title page

\def\fnum{equation}
\newtheorem{thm}[\fnum]{Theorem}
\newtheorem{cor}[\fnum]{Corollary}
\newtheorem{Que}[\fnum]{Question}
\newtheorem{lem}[\fnum]{Lemma}
\newtheorem{Con}[\fnum]{Conjecture}
\newtheorem{Pro}[\fnum]{Proposition}
\theoremstyle{definition}
\newtheorem{defn}[\fnum]{Definition}
\newtheorem{Exa}[\fnum]{Example}
\newtheorem{rem}[\fnum]{Remark}

%\renewcommand{\theThm} {\thesection.\arabic{Thm}}
%\renewcommand{\theProp}{\thesection.\arabic{Prop}}
%\renewcommand{\theLem}{\thesection.\arabic{Lem}}
%\renewcommand{\theCor}{\thesection.\arabic{Cor}}
%\renewcommand{\theDef}{\thesection.\arabic{Def}}
%\renewcommand{\theGuess}{\thesection.\arabic{Guess}}
%\renewcommand{\theEx}{\thesection.\arabic{Ex}}
%\renewcommand{\theRmk}{}
%\renewcommand{\theNot}{}
%\renewcommand{\thefootnote}{\fnsymbol{footnote}}
%\rennewcommand{\thefootnote}{\arabic{footnote}}

% Put abstract here
\newpage
\chapter*{Abstract}
\addcontentsline{toc}{chapter}{\protect\numberline{}Abstract} We investigate Mean Curvature Flow self-shrinking hypersurfaces with polynomial growth.  It is known that such self shrinkers are unstable.  We focus mostly on self-shrinkers of the form $\mathbb S^k\times\R^{n-k}\subset \R^{n+1}$.  We use a connection between the stability operator and the quantum harmonic oscillator Hamiltonian to find all eigenvalues and eigenfunctions of the stability operator on these self-shrinkers.  We also show self-shrinkers of this form have lower index than all other complete self-shrinking hypersurfaces.  In particular, they have finite index.  This implies that the ends of such self shrinkers must be stable.  We look for the largest stable regions of these self shrinkers.

\textsc{Readers:} William P. Minicozzi II (Advisor), Chikako Mese, and Joel Spruck.
% End abstract

\newpage
\chapter*{Acknowledgments}
\addcontentsline{toc}{chapter}{\protect\numberline{}Acknowledgments}
\n I would first and foremost like to thank my advisor, Dr. William Minicozzi, for his patience and advice.  This dissertation could not have been written without him.  I also want to thank my friend and colleague Sinan Ariturk for numerous fruitful discussions.  I owe many thanks to my family and friends for their love, support, and distractions.  Finally, this dissertation is dedicated to Claudia, for keeping me happy and motivated.
\vspace{.1in}

\newpage
\tableofcontents
\listoffigures

\newpage
\pagenumbering{arabic}

\chapter{Introduction}

Mean Curvature Flow (MCF) is a nonlinear second order differential equation, so we are guaranteed the short-time existence of solutions \cite{HP}.  However, unlike the heat equation, solutions of MCF with initially smooth data often develop singularities.  Understanding these singularities and extending MCF past them have historically been the greatest impediments to the study of MCF.

A surface flowing by mean curvature tends to decrease in area.  In fact, it is possible to show that MCF is the $L^2$-gradient flow of the area functional, meaning that in some sense flowing a surface by its mean curvature decreases the surface's area as quickly as possible.  Because of this it makes sense that the surfaces that don't change under MCF are precisely minimal surfaces.  This is also clear from the fact that minimal surfaces have mean curvature $H=0$.

When applied to curves in the plane, MCF is called the curve shortening flow.  The singularities that can develop from the curve shortening flow are well understood thanks to the work of Abresch-Langer, Grayson, Gage-Hamilton, and others.  Grayson showed that any smooth simple closed curve in the plane will evolve under the curve shortening flow into a smooth convex curve \cite{G1}.  It had already been proven by Gage and Hamilton that any smooth convex curve will remain smooth until it disappears in a round point \cite{GH}.

Abresch and Langer found a family of non-embedded plane curves that evolve under the curve shortening flow by holding their shape and shrinking until they become extinct at singularities \cite{AL}.  These curves remain non-embedded until they become extinct, meaning that these singularities are different from the "round points" arising from simple closed curves evolving under the curve shortening flow.

In both of the examples above, singularities can be modeled by self-shrinking curves, that is by curves that change by dilations when they flow by mean curvature.  It follows from the work of Huisken in \cite{H} (see Section \ref{58} below) that as a hypersurface flowing by mean curvature in any dimension approaches a singularity, it asymptotically approaches a self-shrinker.  Thus, there is great interest in understanding the possible MCF self-shrinkers, since these describe all the possible singularities of MCF.

Colding and Minicozzi have recently proven that the only mean curvature flow singularities that cannot be perturbed away are the sphere and generalized cylinders given by $\s^k\times\R^{n-k}$ with $1\leq k\leq n$ \cite{CM1}.  When $k=0$ the product space above is simply a hyperplane, which is a self-shrinker but of course does not give rise to any singularity.  In this dissertation we focus on these same self-shrinkers, allowing the $k=0$ hyperplane case.

In Chapter \ref{59} we define MCF and self-shrinkers and discuss known examples of self-shrinkers.  We also define what we mean by stability of a self-shrinker and present some necessary background results.

In Chapter \ref{26} we find the spectrum and the eigenfunctions of the stability operator on all self-shrinking hypersurfaces of the form $\s^k\times\R^{n-k}$.  We also prove that these self-shrinkers have strictly lower stability index than all other self-shrinkers.

In Chapter \ref{24} we investigate the largest stable regions of these generalized cylinders using two techniques.  First, we apply the results of Chapter \ref{26} to find several stable regions, using the fact that a stability operator eigenfunction with eigenvalue $0$ is a Jacobi function.  Then we look for stable regions with a rotational symmetry.

In Chapter \ref{60} we solve a minimization problem to obtain stable portions of a family of surfaces recently discovered by Kleene and M{\o}ller \cite{KM}.  We use the stability of these regions to find a larger stable region of the plane in $\R^3$ than we were able to find in Chapter \ref{24}.

In Chapter \ref{61} we discuss a few open questions for future research.

\chapter{Background}\label{59}

\section{Preliminaries}

Throughout, we will assume the following.  Let $\Sigma^n\subset \R^{n+1}$ be an orientable connected and differentiable hypersurface with unit normal $\vec{n}$.  We will only consider the case where $n\geq 2$.  When the dimension of $\Sigma^n$ is clear or unimportant, we will sometimes omit the $n$ and write simply $\Sigma$.  When possible, we will choose the unit normal $\vec{n}$ to be outward pointing.  The mean curvature $H$ of $\Sigma$ is defined by $H=\text{div} (\vec{n})$.  We will only consider hypersurfaces $\Sigma$ with polynomial volume growth.  This means that there exists a point $p\in\Sigma$ and constants $C$ and $d$ such that for all $r\geq1$ $$Vol(B_r(p)\cap\Sigma)\leq C r^d$$ where $B_r(p)$ denotes the ball of radius $r$ about the point $p$.

We let $A$ denote the second fundamental form on $\Sigma$.  If $\{\vec{e_i}\}_{i=1}^n$ is an orthonormal frame for $\Sigma$, the components of $A$ with respect to $\{\vec{e_i}\}_{i=1}^n$ are given by
\begin{equation}\label{34}
a_{ij}=\langle \nabla_{\vec{e_i}}\vec{e_j},\vec n\rangle
\end{equation}
Differentiating $\langle\vec{e_j},\vec n\rangle=0$, we obtain
\begin{equation}\label{35}
\nabla_{\vec{e_i}}\vec n=-a_{ij}\vec{e_j}
\end{equation}
Thus the mean curvature $H=-a_{ii}$.  Here, as in the rest of the paper, we use the convention that repeated indices are summed over.

With this definition, the mean curvature of a sphere of radius $r$ is given by $H=\frac 2 r$, and the mean curvature of a cylinder of radius $r$ is given by $H=\frac 1 r$ \cite{CM1}.  Note that with this definition, the mean curvature is the sum of the principal curvatures, rather than their average.

Unless otherwise stated, $\langle\cdot,\cdot\rangle$ denotes the Euclidean inner product.  We also let $B^{\Sigma}_r(p)$ denote the geodesic ball in $\Sigma$ about the point $p\in\Sigma$.  In other words, $B^{\Sigma}_r(p)$ denotes the set of all points in $\Sigma$ whose intrinsic distance from $p$ is less than $r$.  The standard ball of radius $r$ about the point $p$ in Euclidean $n$-space will be denoted by $B^n_r(p)$, or simply by $B_r(p)$ when the dimension is clear from context.

\section{Mean Curvature Flow (MCF)}\label{58}

The idea of mean curvature flow is to start with a surface $\Sigma$, and at every point move $\Sigma$ in the direction $-\vec{n}$ with speed $H$.  Note that for a hypersurface this process does not depend on our choice of $\vec n$, since replacing $\vec n$ by its opposite will change the sign of $H=\text{div}(\vec n)$.

\begin{defn}\cite{CM1} Let $\Sigma_t$ be a family of hypersurfaces in $\R^3$.  We say $\Sigma_t$ flows by mean curvature if $$(\partial_t \vec{x})^\bot=-H\vec{n}.$$  Here $(\partial_t \vec{x})^\bot$ denotes the component of $\partial_t \vec{x}$ which is perpendicular to $\Sigma_t$ at $\vec{x}$.
\end{defn}

This restriction to the normal component is reasonable, because tangential components of $\partial_t\vec{x}$ correspond to internal reparameterizations of $\Sigma_t$, which do not affect the geometry of the surface.  Indeed, by changing the parameterizations of $\Sigma_t$, $\partial_t\vec{x}$ can be made perpendicular to $\Sigma_t.$

Note that minimal surfaces satisfy $H=0$ and are thus invariant under MCF.  If $\Sigma$ is a minimal surface, and $\Sigma_t$ is any smooth family of parameterizations of $\Sigma$, then $\Sigma_t$ flows by mean curvature.

We now compute the simplest non-static solution of MCF.  This is a round sphere which shrinks until it disappears in finite time at a singularity.  An equivalent computation would work for hyperspheres in arbitrary dimension.

\begin{Exa}\label{1} Consider the case of a round sphere of radius $r$ in $\R^3.$  At every point the mean curvature of this sphere is given by $H=\frac 2 r$.  As $r\rightarrow 0$, $H\rightarrow \infty$.  Since $(\partial_t\vec x)^\bot=-H\vec n$, the sphere is shrinking, and the rate at which the sphere shrinks grows as $r\rightarrow 0$.  Thus, the sphere becomes extinct in finite time.

By the symmetry of the sphere, it is clear that as the surface flows by mean curvature it remains a sphere until it reaches the singularity and vanishes.  Thus we can model the evolution of the surface by an ordinary differential equation for the radius $r=r(t)$.  $$(\partial_t\vec x)^\bot=\frac{dr}{dt}\vec{n}=-H\vec{n}$$  Substituting in for $H$ yields $$\frac{dr}{dt}=-\frac 2 r$$  which has the solution $r=\sqrt{c-4t}$.

If we want to determine the constant of integration $c$, then we need to choose a convention.  Clearly the point where $r=0$ corresponds to the singularity, so specifying the time at which the singularity occurs is equivalent to choosing the constant $c$.  It is standard to choose the constant so that the singularity occurs at time $t=0$.  That makes the constant 0, so our solution becomes $r=\sqrt{-4t}$.  Note that at the time $t=-1$ our sphere has radius 2.
\end{Exa}

MCF is a nonlinear parabolic flow, so its solutions satisfy a maximum principle \cite{N}.  This means that if two complete hypersurfaces are initially disjoint, then they will remain disjoint when allowed to flow by mean curvature.  This implies via the following argument that the development of singularities in MCF is a widespread phenomenon.

Let $\Sigma$ be any closed surface in $\R^3$, meaning that $\Sigma$ is compact and has no boundary.  Then there exists some $r$ such that $\Sigma$ is completely contained in $B_r(0)$.  Letting $S_r=\partial B_r(0)$ be the boundary of this ball, we see that $S_r$ and $\Sigma$ are disjoint.  Then by the maximum principle they must stay disjoint when they flow by mean curvature.  However, by Example \ref{1} we know $S_r$ becomes extinct in finite time.  Thus, $\Sigma$ must become extinct at a singularity before $S_r$ does.  Thus, all closed surfaces in $\R^3$ become singular in finite time.

We now prove a well known theorem with some important implications for the study of MCF singularities.  We follow the conventions used in \cite{CM1}.

\begin{thm}[Huisken's Monotonicity Formula]\cite{H}
Let $\Sigma_t$ be a surface flowing by mean curvature.  Define $\Phi (\vec x,t)=[-4\pi t]^{-\frac n 2 }e^{\frac {|\vec x|^2}{4t}}$, and set $\Phi_{(\vec {x_0},t_0)}=\Phi(\vec x-\vec {x_0},t-t_0)$.  Then $$\frac d{dt}\int_{\Sigma_t}\Phi_{(\vec {x_0},t_0)}d\mu=-\int_{\Sigma_t} \Bigl| H\vec{n}-\frac{(\vec x-\vec {x_0})^\bot}{2(t_0-t)}\Bigr|^2\Phi_{(\vec {x_0},t_0)}d\mu.$$
\end{thm}
\begin{proof}\cite{H}  We begin by setting $\vec y=\vec x-\vec {x_0}$ and $\tau=t-t_0$.  We will also set $\tilde\Phi = \Phi_{(\vec {x_0},t_0)}$ to simplify notation.  Then when we differentiate $\int_{\Sigma_t}\tilde\Phi d\mu$ we obtain an extra term coming from how $d\mu$ changes as the surface flows as a function of $t$.  In general (see e.g. \cite{CM1}), if $f\vec n$ is a variation of $\Sigma$, then $$(d\mu)'=fH d\mu.$$

Since $\Sigma$ is flowing by mean curvature, we have that $(d\mu)'=-H^2d\mu$.  Now we compute.
\begin{align} \tilde\Phi \notag &= (-4\pi\tau)^{-\frac n 2}e^{\frac{|\vec y|^2}{4\tau}} \\
\frac d{dt}\int\limits_{\Sigma_t}\tilde\Phi d\mu \notag &= \int\limits_{\Sigma_t} \left[\frac d{dt}\tilde\Phi - H^2\tilde\Phi  \right]d\mu \\
\frac d{dt}\int\limits_{\Sigma_t}\tilde\Phi d\mu \notag &= \int\limits_{\Sigma_t} \left[\frac {\partial}{\partial t}\tilde\Phi +\left\langle \nabla\tilde\Phi,\frac{d\vec x}{dt} \right\rangle -H^2\tilde\Phi \right]d\mu \\
\frac{d\vec y}{dt} \notag &= \frac{d\vec x}{dt} = -H\vec n \\
\frac d{dt}\int\limits_{\Sigma_t}\tilde\Phi d\mu \notag &= \int\limits_{\Sigma_t} \left[-\frac n{2\tau}-\frac{|\vec y|^2}{4\tau^2}-\frac{H}{2\tau}\langle\vec y,\vec n\rangle -H^2  \right]\tilde\Phi d\mu \\
\frac d{dt}\int\limits_{\Sigma_t}\tilde\Phi d\mu \notag &= -\int\limits_{\Sigma_t}\tilde\Phi \left| H\vec n+\frac{\vec y}{2\tau} \right|^2 d\mu + \int\limits_{\Sigma_t} \frac{H}{2\tau}\langle \vec y,\vec n\rangle\tilde \Phi d\mu -\int\limits_{\Sigma_t}\frac n{2\tau}\tilde\Phi d\mu
\end{align}

However, for any vector field $\vec v$ we can write $\text {div}_{\Sigma} \vec v=\text {div}_\Sigma \vec v^\top +H\langle\vec v,\vec n\rangle$.  We fix a value of $t$ and consider $\Sigma(r)=B_r(p)\cap\Sigma_t$.  Note that as $r\rightarrow \infty$, $\Sigma(r)\rightarrow\Sigma_t$.  We then choose a specific vector field $\vec v$ and apply Stokes' Theorem.  We let $\vec \nu$ denote the unit vector field normal to $\partial \Sigma(r)$ and tangent to $\Sigma(r)$, and let $d\tilde\mu$ denote the measure on $\partial \Sigma(r)$.
\begin{align} \vec v \notag &= \frac 1 {2\tau}\tilde\Phi\vec y \\
\int\limits_{\Sigma(r)}\text {div}_\Sigma \vec v^\top d\mu \notag &= \int\limits_{\partial\Sigma(r)} \langle \vec v,\vec\nu\rangle d\tilde\mu \\
\int\limits_{\Sigma(r)}\text {div}_\Sigma \vec v^\top d\mu \notag &= \int\limits_{\partial\Sigma(r)} \frac 1 {2\tau}\tilde\Phi\langle \vec y,\vec\nu\rangle d\tilde\mu \\
\int\limits_{\Sigma_t}\text {div}_\Sigma \vec v^\top d\mu \notag &= \lim\limits_{r\rightarrow\infty} \int\limits_{\partial\Sigma(r)} \frac 1 {2\tau}\tilde\Phi\langle \vec y,\vec\nu\rangle d\tilde\mu \\
\int\limits_{\Sigma_t}\text {div}_\Sigma \vec v^\top d\mu \notag &= 0
\end{align}

In the last equality we used the fact that $\Sigma_t$ has polynomial volume growth, so the exponential decay of $\tilde\Phi$ dominates in the limit.  We let $\{\vec {e_i}\}_{i=1}^n$ be an orthonormal frame for $\Sigma_t$ and continue.
\begin{align}\int\limits_{\Sigma_t}\text {div}_{\Sigma} \left(\frac 1 {2\tau}\tilde\Phi\vec y\right)d\mu \notag &=\int\limits_{\Sigma_t}H\langle\frac 1 {2\tau}\tilde\Phi\vec y,\vec n\rangle d\mu \\
\int\limits_{\Sigma_t} \frac{H}{2\tau}\tilde\Phi \langle \vec y,\vec n\rangle d\mu \notag &= \int\limits_{\Sigma_t}\frac 1{2\tau}\text {div}_{\Sigma} (\tilde\Phi\vec y)d\mu \\
\nabla_{\vec{e_i}} (\tilde\Phi\langle\vec y,\vec{e_i}\rangle) \notag &= \tilde\Phi \langle\vec {e_i},\vec{e_i}\rangle + \tilde\Phi \langle\vec y,\vec{e_i}\rangle \frac{\langle\vec y,\vec{e_i}\rangle}{2\tau} \\
\frac 1{2\tau}\text {div}_{\Sigma} (\tilde\Phi\vec y)d\mu \notag &= \frac n{2\tau}\tilde\Phi + \frac{|\vec y^\top|^2}{4\tau^2}\tilde\Phi \\
\int\limits_{\Sigma_t} \frac{H}{2\tau}\tilde\Phi \langle \vec y,\vec n\rangle d\mu \notag &= \int\limits_{\Sigma_t}\left(\frac n{2\tau} + \frac{|\vec y^\top|^2}{4\tau^2}  \right)\tilde\Phi d\mu \\
\frac d{dt}\int\limits_{\Sigma_t}\tilde\Phi d\mu \notag &= -\int\limits_{\Sigma_t}\tilde\Phi \left| H\vec n+\frac{\vec y}{2\tau} \right|^2 d\mu + \int\limits_{\Sigma_t} \frac{|\vec y^\top|^2}{4\tau^2} \tilde\Phi d\mu \\
\frac d{dt}\int\limits_{\Sigma_t}\tilde\Phi d\mu \notag &= -\int\limits_{\Sigma_t}\tilde\Phi \left| H\vec n+\frac{\vec y^\bot}{2\tau} \right|^2 d\mu
\end{align}
\end{proof}

Understanding the possible singularities of MCF has been one of the main obstacles to the theory, because these are the places where the MCF ceases to be smooth.  Also, a connected surface flowing by mean curvature can split into multiple connected components at a singularity.  As an example of this, Grayson \cite{G} constructed a dumbbell that flows smoothly by mean curvature until it splits into two topological spheres (see Subsection \ref{63} below).

However, using Huisken's monotonicity formula it is possible to show \cite{H} that in the limit as a surface evolving by MCF approaches a singularity the surface asymptotically approaches a self-shrinker, which we will define in the next chapter.  Thus, a key step in understanding the singularities of MCF is to understand its self-shrinkers.

\section{Self-Shrinkers}

The sphere discussed in Example \ref{1} is a self-shrinker, meaning that as it flows by mean curvature it changes by a dilation.  However, as we did in the example, we will often specify that the MCF should become extinct at $t=0$ and $\vec{x}=0$.  In such cases, we will refer to a surface $\Sigma$ as a self-shrinker only if it is the $t=-1$ time slice of such a MCF that will become extinct at $t=0$ and $\vec{x}=0$.  So for instance, by the computation in Example \ref{1}, we would consider a sphere to be a self-shrinker only if it has radius 2 and is centered at the origin.  Otherwise, it would be a different time slice of the MCF in question, or if it is not centered at the origin it would become extinct at a point other than $\vec{x}=0$.  This leads us to the following definition.

\begin{defn}\cite{CM1} \label{7}  A surface $\Sigma$ is a self-shrinker if the family of surfaces $\Sigma_t=\sqrt{-t}\Sigma$ flows by mean curvature.
\end{defn}

We will also need the following two equivalent characterizations of self-shrinkers.

\begin{lem}\label{30} Suppose a surface $\Sigma^n\subset\R^{n+1}$ satisfies

 \begin{equation}\label{2}
 H=\frac{\langle \vec{x},\vec{n}\rangle}2.
 \end{equation}
 Then $\Sigma_t=\sqrt{-t}\Sigma$ flows by mean curvature, and for every $t$ we have
  \begin{equation}\label{36}
 H_{\Sigma_t}=-\frac{\langle \vec{x},\vec{n}_{\Sigma_t}\rangle}{2t}.
 \end{equation}

Conversely, suppose a family of surfaces $\Sigma_t$ flows by mean curvature.  Then $\Sigma_t$ is a self-shrinker, meaning $\Sigma_t=\sqrt{-t}\Sigma_{-1}$, if and only if $\Sigma_t$ satisfies \ref{36}.
\end{lem}
\begin{proof} Suppose $\Sigma$ has mean curvature $H=\frac 1 2 \langle \vec x,\vec n\rangle$.  We wish to show $\Sigma_t=\sqrt{-t}\Sigma$ flows by mean curvature.  For each $\vec p\in\Sigma$ and $t\in(-\infty,0)$, set $\vec x(\vec p, t)=\sqrt{-t}\vec p$.  Then
\begin{align}\vec n_{\Sigma_t}(\vec x(\vec p,t)) \notag &= \vec n_{\Sigma}(\vec p) \\
H_{\Sigma_t}(\vec x(\vec p,t)) \notag &= \frac 1{\sqrt{-t}}H_\Sigma(\vec p)
\end{align}

To see this scaling rule for the mean curvature, consider a circle in $\R^2$.  Scaling the circle by $\lambda$ causes the radius to scale by $\lambda$, which causes the curvature to scale by $\frac 1 {\lambda}$.  This example is actually completely general, because the mean curvature of $\Sigma$ equals the sum of the principal curvatures of $\Sigma$.  Each principal curvature can be defined as the inverse of the radius of the circle of best fit to a curve on $\Sigma$.

We compute $H_{\Sigma_t}$.
\begin{align} \partial_t\vec x \notag &= \partial_t(\sqrt{-t}\vec p) \\
\partial_t\vec x \notag &= -\frac 1 {2\sqrt {-t}}\vec p \\
(\partial_t\vec x)^\bot \notag &= -\frac {\langle \vec p,\vec n_{\Sigma}\rangle} {2\sqrt {-t}} \\
(\partial_t\vec x)^\bot \notag &= -H_{\Sigma_t}
\end{align}
Thus $\Sigma_t$ flows by mean curvature.  We now show that $\Sigma_t$ also satisfies Equation \ref{36}.  By the above
\begin{align} H_{\Sigma_t} \notag &= \frac {\langle \vec p,\vec n_{\Sigma}\rangle} {2\sqrt {-t}} \\
H_{\Sigma_t} \notag &= \frac {\langle \sqrt{-t}\vec p,\vec n_{\Sigma_t}\rangle} {-2t} \\
H_{\Sigma_t} \notag &= -\frac {\langle \vec x,\vec n_{\Sigma_t}\rangle} {2t}
\end{align}

Conversely, suppose that $\Sigma_t$ flows by mean curvature.  Suppose first that $\Sigma_t$ satisfies Equation \ref{36} for all $t$.  Then setting $t=-1$ yields Equation \ref{2}, so the above argument shows that $\Sigma_t$ is a self-shrinker.  Now suppose instead that $\Sigma_{-1}=\frac{\Sigma_t}{\sqrt{-t}}$.  Then note that
\begin{align} (-t)^{\frac 3 2}\partial_t \left(\frac {\vec x}{\sqrt{-t}}\right) \notag &= (-t)^{\frac 3 2}\left(\frac{\partial_t\vec x}{\sqrt{-t}} +\frac{\vec x}{2(-t)^{\frac 3 2}} \right) \\
(-t)^{\frac 3 2}\partial_t \left(\frac {\vec x}{\sqrt{-t}}\right) \notag &= -t\partial_t\vec x+\frac{\vec x}2
\end{align}
However, $\dfrac {\vec x}{\sqrt {-t}}$ is fixed, so $\partial_t\left(\dfrac {\vec x}{\sqrt {-t}}\right)=0$.
\begin{align} 0 \notag &= (-t)^{\frac 3 2}\left\langle \partial_t\left(\frac {\vec x}{\sqrt{-t}}\right),\vec n_{\Sigma_{-1}}\right\rangle \\
0 \notag &= -t\langle \partial_t\vec x,\vec n_{\Sigma_{-1}}\rangle +\frac 1 2 \langle\vec x,\vec n_{\Sigma_{-1}}\rangle
\end{align}
By assumption $\Sigma_t$ flows by mean curvature, so
\begin{align} H_{\Sigma_{-1}} \notag &= -\langle \partial_t \vec x, \vec n_{\Sigma_{-1}}\rangle \\
H_{\Sigma_{-1}} \notag &= \frac 1 2\langle\vec x, \vec n_{\Sigma_{-1}}\rangle
\end{align}
Then $\Sigma_{-1}$ satisfies Equation \ref{2}, so $\Sigma_t$ must satisfy Equation \ref{36}
\end{proof}

We now define a two parameter family of functionals that are closely related to self-shrinkers.  These functionals will help us discuss the stability of self-shrinkers.

\begin{thm}\cite{CM1}\label{11}  Let $\vec{x_0}\in\R^3$, and let $t_0>0$.  Then given a surface $\Sigma\subset\R^n$ define the functional $F_{\vec{x_0},t_0}$ by
$$F_{\vec{x_0},t_0}(\Sigma)=(4\pi t_0)^{-\frac n 2}\int_\Sigma e^{\frac{-|\vec x-\vec{x_0}|^2}{4t_0}}d\mu.$$  Then $\Sigma$ is a critical point of $F_{\vec{x_0},t_0}$ if and only if $\Sigma=\Sigma_{-t_0}$ where $\Sigma_t=\vec{x_0} + \sqrt{-t}\Sigma_{-1}$ is flowing by mean curvature.  Note that at $t=0$, $\Sigma_t$ becomes extinct at $\vec{x_0}$.
\end{thm}
\begin{proof} Let $\Sigma_s$ be a variation of $\Sigma$ with variation vector field $\frac d {ds}\big|_{s=0} \vec x=f\vec n$.  Recall from the proof of Huisken's Monotonicity Formula that $(d\mu)'=fH d\mu.$  Then we compute.

\begin{align}\frac d{ds}\bigg|_{s=0} F_{\vec{x_0},t_0}(\Sigma_s) \notag &= (4\pi t_0)^{-\frac n 2}\int_\Sigma \left[fHe^{\frac{-|\vec x-\vec{x_0}|^2}{4t_0}} +\frac d{ds}\bigg|_{s=0}e^{\frac{-|\vec x-\vec{x_0}|^2}{4t_0}}\right]d\mu \\
\frac d{ds}\bigg|_{s=0} F_{\vec{x_0},t_0}(\Sigma_s) \notag &= (4\pi t_0)^{-\frac n 2}\int_\Sigma \left[fH+ \frac d{ds}\bigg|_{s=0}\log e^{\frac{-|\vec x-\vec{x_0}|^2}{4t_0}}\right]e^{\frac{-|\vec x-\vec{x_0}|^2}{4t_0}} d\mu \\
\frac d{ds}\bigg|_{s=0} F_{\vec{x_0},t_0}(\Sigma_s) \notag &= (4\pi t_0)^{-\frac n 2}\int_\Sigma \left[fH - \frac d{ds}\bigg|_{s=0}\frac{\langle\vec x-\vec{x_0}, \vec x-\vec{x_0}\rangle}{4t_0}\right]e^{\frac{-|\vec x-\vec{x_0}|^2}{4t_0}} d\mu \\
\frac d{ds}\bigg|_{s=0} F_{\vec{x_0},t_0}(\Sigma_s) \notag &= (4\pi t_0)^{-\frac n 2}\int_\Sigma \left[fH - \frac{\langle\vec x-\vec{x_0}, f\vec n\rangle}{2t_0}\right]e^{\frac{-|\vec x-\vec{x_0}|^2}{4t_0}} d\mu \\
\frac d{ds}\bigg|_{s=0} F_{\vec{x_0},t_0}(\Sigma_s) \notag &= (4\pi t_0)^{-\frac n 2}\int_\Sigma \left[H - \frac{\langle\vec x-\vec{x_0}, \vec n\rangle}{2t_0}\right]f e^{\frac{-|\vec x-\vec{x_0}|^2}{4t_0}} d\mu
\end{align}

Thus, $\Sigma$ is a critical point of $F_{\vec{x_0},t_0}$ if and only if $H = \frac{\langle\vec x-\vec{x_0}, \vec n\rangle}{2t_0}$.  The result follows from Lemma \ref{30} and changing coordinates.
\end{proof}

Thus, by Lemma \ref{30}, a surface is a self-shrinker if and only if it is a critical point of the functional $F_{0,1}$.

\subsection{Examples of Self-Shrinkers}\label{63}

There are several standard examples of self-shrinkers.  In Example \ref{1} we explicitly computed the mean curvature flow of a round sphere in $\R^3$ to show that a $2$-sphere of radius $2$ centered at the origin is a self-shrinker.  In general a round sphere $\mathbb S^n\subset\R^{n+1}$ is a self-shrinker if and only if it is centered at the origin and has radius $\sqrt{2n}$.  To see this, we apply Lemma \ref{30}.

Note that an $n$-sphere of radius $r$ has mean curvature $$H=\sum\limits_{i=1}^n\frac 1 r=\frac n r.$$  This is constant, and in order for $\langle\vec x,\vec n\rangle$ to be constant the sphere must be centered at the origin.  In this case $\vec x=r\vec n$.
\begin{align} \frac {\langle\vec x,\vec n\rangle} 2 \notag &= \frac r 2 \\
H \notag &= \frac {\langle\vec x,\vec n\rangle} 2 \\
\frac n r \notag &= \frac r 2 \\
r^2 \notag &= 2n
\end{align}
and the claim follows.

Recall that minimal surfaces are fixed points of MCF, since they satisfy $H\equiv 0$.  However, in general minimal surfaces do not satisfy $$\sqrt{-t}\Sigma=\Sigma.$$  In order to be invariant under dilations and hence a self-shrinker, a minimal surface $\Sigma$ must be a cone.  Then in order to be a smooth surface $\Sigma$ must be a flat plane through the origin.

We can also generalize the above two examples to show that $\Sigma^n=\mathbb S^k\times\R^{n-k}\subset\R^{n+1}$ is a self-shrinker if the $\mathbb S^k$ factor is a round $k$-sphere centered at the origin with radius $\sqrt{2k}$.  This will follow from Lemma \ref{41} below, or we can again apply Lemma \ref{30}.
\begin{align} H \notag &= \frac k{\sqrt{2k}} = \sqrt{\frac k 2} \\
\frac{\langle\vec x,\vec n\rangle}2 \notag &= \frac {\langle \sqrt{2k}\vec n,\vec n\rangle}2 = \frac{\sqrt{2k}} 2 = \sqrt {\frac k 2} \\
H \notag &= \frac {\langle\vec x,\vec n\rangle}2
\end{align}
Thus $\Sigma$ is a self-shrinker.

There is also a family of embedded self-shrinkers that are topologically $\mathbb S^1\times\mathbb S^{n-1}\subset\R^{n+1}$ for each $n\geq 2$.  These surfaces were constructed by Angenent in \cite{Ang}.  When $n=2$, the obtained surface is a self-shrinking torus.  Angenent used this torus to give another proof that Grayson's dumbbell \cite{G} splits at a singularity into two topological spheres.

The idea is to enclose the neck of the dumbbell by the self-shrinking torus and make the bells on either side large enough to each enclose a large sphere.  Make the enclosed spheres large enough that they will not become extinct until after the self-shrinking torus does.  Then the maximum principle gives that none of the surfaces can intersect any of the others, so before the torus becomes extinct the neck of the dumbbell must pinch off, leaving two topological spheres.

We mentioned above that as a surface flowing by mean curvature approaches a singularity it asymptotically approaches a self-shrinker \cite{H}.  In the case of the dumbbell described above, the singularity is of the type of a self-shrinking cylinder.  

Not all self-shrinkers are embedded.  Abresch and Langer \cite{AL} constructed self-intersecting curves in the plane that are self-shrinkers under MCF.  It is possible to obtain self-shrinking non-embedded surfaces by taking products of these curves with Euclidean factors (see Lemma \ref{41} below).  However, it is noted in \cite{CM1} that the assumption of embeddedness in the classification of generic singularities is made unnecessary by the work of Epstein and Weinstein \cite{EW}.

It should also be noted that there are more complicated self-shrinkers that can arise as singularities.  For example, Kapouleas, Kleene and M{\o}ller \cite{KKM} recently proved the existence of noncompact self-shrinkers of arbitrary genus, as long as that genus is sufficiently large.

\section{Two Types of Stability for Self-Shrinkers}

Now that we have characterized self-shrinkers as the critical points of a functional (see Theorem \ref{11}), we can define stability in the usual manner.  By taking the second variation of the functional, we obtain the following stability operator.

\begin{defn}\cite{CM1}\label{8} Define the stability operator on a self-shrinking hypersurface $\Sigma^n\subset\R^{n+1}$ as \begin{equation}\label{3}
Lf=\Delta f-\frac 1 2 \langle \vec{x},\nabla f\rangle+(|A|^2+\frac 1 2)f\end{equation}
where $|A|^2$ denotes the norm squared of the second fundamental form.  Here $\Delta$ denotes the Laplacian on the manifold $\Sigma$.  Likewise, $\nabla$ denotes the tangential gradient, rather than the full Euclidean gradient.  We call any function $u:\Sigma\rightarrow\R$ such that $Lu=0$ a Jacobi function.  We say $\Sigma$ is stable if there exists a positive Jacobi function on $\Sigma$.
\end{defn}

\begin{thm}\cite{CM1}\label{12} A self-shrinker $\Sigma$ is stable if and only if it is a local minimum of the $F_{0,1}$ functional.  That is, let $\tilde{\Sigma}$ be an arbitrary graph over $\Sigma$ of a function with sufficiently small $C^2$ norm.  Then $\Sigma$ is stable if and only if $F_{0,1}(\Sigma)\leq F_{0,1}(\tilde{\Sigma})$ for every such $\tilde{\Sigma}$.
\end{thm}

However, it turns out that with this standard definition, every complete self-shrinker is unstable.  This is because it is always possible to decrease the functional $F_{0,1}$ by translating $\Sigma$ in either space or time.  Since $\Sigma$ is a self-shrinker, by ``translating in time'' we mean dilating $\Sigma$.  The following notion of F-stability compensates for this inherent instability by defining a self-shrinker to be F-stable if translations in space and time are the only local ways to decrease the $F_{0,1}$ functional.

\begin{defn}\cite{CM1} We say a self-shrinker $\Sigma$ is F-stable if for every normal variation $f\vec{n}$ of $\Sigma$ there exist variations $x_s$ of $x_0$ and $t_s$ of $t_0$ that make $F''=(F_{x_s,t_s}(\Sigma+sf\vec{n}))''\geq 0$ at $s=0$.
\end{defn}

In this dissertation we will focus exclusively on the standard notion of stability from Definition \ref{8}.

\begin{defn}\label{10} If $\tilde L$ is any operator on a surface $\Sigma$, then we say $u$ is an eigenfunction of $\tilde L$ if $u\in L^2(\Sigma)$, $u$ is not identically $0$, and $$\tilde L u=-\lambda u$$ for some constant $\lambda$.  In this case we say that $\lambda$ is an eigenvalue of $\tilde L$ corresponding to the eigenfunction $u$.
\end{defn}

\begin{defn}\cite{FC}\label{16} Let $\Sigma$ be a self-shrinker, and let $L$ be the stability operator on $\Sigma$.  The index of $L$ on $\Sigma$ is defined to be the supremum over compact domains of $\Sigma$ of the number of negative eigenvalues of $L$ with 0 boundary data.  We will also call this the stability index of $\Sigma$.
\end{defn}

Notice that a self-shrinker is stable if and only if its stability index is 0.

\begin{thm}\label{45}\cite{FC} A self-shrinker with finite stability index is stable outside of some compact set.
\end{thm}
\begin{proof}
Fix a point $0\in \Sigma$, and let $B^\Sigma_\rho (0)$ denote the intrinsic ball of radius $\rho$ centered at 0.  We will use the fact that $L$ is stable on any set of sufficiently small area.

Thus, for $\rho$ sufficiently small $L$ is stable on $B_\rho(0)$.  Let $\rho_1=2\sup\{\rho|L\text{ is stable on } B_\rho(0)\}$.  If $\rho_1$ is infinite, then $L$ is stable on all of $\Sigma$ and we are done.  Otherwise let $\rho_2=2\sup\{\rho>\rho_1|L\text{ is stable on } B_\rho(0)\setminus B_{\rho_1}(0)\}$.  If $\rho_2$ is infinite, then $L$ is stable on $\Sigma\setminus B_{\rho_1}(0)$ and we are done.  Otherwise repeat this construction to obtain $\{\rho_k\}$ such that $L$ is strictly unstable on each $B_{\rho_k}(0)\setminus B_{\rho_{k-1}}(0)$.

$L$ has at least one negative eigenvalue for each $B_{\rho_k}(0)\setminus B_{\rho_{k-1}}(0)$ (assuming $\rho_k<\infty$), so let $f_k$ denote a corresponding eigenfunction on $B_{\rho_k}(0)\setminus B_{\rho_{k-1}}(0)$.  The $f_k$ are defined on sets which are disjoint except for sets of measure 0, so the $f_k$ are independent.  By assumption the index of $L$ is finite, so there must be only finitely many $f_k$.  Thus some $\rho_{k+1}$ is infinite, and $L$ is stable outside the compact set $C=\overline{B_{\rho_k}(0)}$.
\end{proof}

\newpage

\chapter{Low Index Self-Shrinkers in Arbitrary Dimension}\label{26}

\section{Self-Shrinkers Splitting off a Line}

This chapter is devoted to proving the following theorem.

\begin{thm}\label{37} Let $\Sigma^n\subset\R^{n+1}$ be a smooth complete embedded self-shrinker without boundary and with polynomial volume growth.

If $\Sigma=\R^n$ is a flat hyperplane through the origin, then the index of $\Sigma$ is $1$.

If $\Sigma=\mathbb S^k\times\R^{n-k}$ for some $1\leq k\leq n$, then the index of $\Sigma$ is $n+2$.

If $\Sigma\neq \mathbb S^k\times\R^{n-k}$, then the index of $\Sigma$ is at least $n+3$.

If $\Sigma\neq \mathbb S^k\times\R^{n-k}$ and $\Sigma$ also splits off a line, then the index of $\Sigma$ is at least $n+4$.
\end{thm}

Note that in the above theorem each $\mathbb S^k$ must have radius $\sqrt{2k}$ in order for $\Sigma^n$ to be a self-shrinker.  We first prove a theorem of Colding and Minicozzi that provides possible eigenfunctions of $L$ on self-shrinkers.

\begin{thm}\label{38}\cite{CM1} Suppose $\Sigma^n\subset\R^{n+1}$ is a smooth self-shrinking hypersurface without boundary.  Then the mean curvature $H$ satisfies $LH=H$.  Also, for any constant vector field $\vec v$ we have $L\langle\vec v,\vec n\rangle=\frac 1 2 \langle\vec v,\vec n\rangle.$  In particular, any of these functions which is not identically $0$ and which is in $L^2(\Sigma)$ must be an eigenfunction with negative eigenvalue.
\end{thm}

\begin{proof}\cite{CM1} Let $\vec p\in\Sigma$ be an arbitrary point, and let $\{\vec {e_i}\}_{i=1}^n$ be an orthonormal frame for $\Sigma$ such that $\nabla_{\vec{e_i}}^{\top}\vec{e_j}(\vec p)=0$.

We begin by showing $LH=H$.  This will show that if $H\in L^2(\Sigma)$ is not identically $0$, then $H$ is an eigenfunction of $L$ with eigenvalue $-1$.  By Equation \ref{2}
\begin{align} H \notag &=\frac 1 2 \langle\vec x,\vec n\rangle \\
\nabla_{\vec{e_i}}H \notag &= \frac 1 2 [\langle \nabla_{\vec{e_i}}\vec x,\vec n\rangle+\langle\vec x,\nabla_{\vec{e_i}}\vec n\rangle ] \\
\nabla_{\vec{e_i}}H \notag &= \frac 1 2 [\langle\vec{e_i},\vec n\rangle-a_{ij}\langle\vec x,\vec{e_j}\rangle]
\end{align}

The last line follows from Equation \ref{35}.  We note that $\langle\vec{e_i},\vec n\rangle=0$, so we continue.
\begin{align} \nabla_{\vec{e_i}}H \notag &= -\frac 1 2 a_{ij}\langle\vec x,\vec{e_j}\rangle \\
\nabla_{\vec{e_k}}\nabla_{\vec{e_i}}H \notag &= -\frac 1 2 a_{ij,k}\langle\vec x,\vec{e_j}\rangle-\frac 1 2 a_{ij}\langle\nabla_{\vec{e_k}}\vec x,\vec{e_j}\rangle-\frac 1 2 a_{ij}\langle\vec x,\nabla_{\vec{e_k}}\vec{e_j}\rangle \\
\nabla_{\vec{e_k}}\nabla_{\vec{e_i}}H \notag &= -\frac 1 2 a_{ij,k}\langle\vec x,\vec{e_j}\rangle-\frac 1 2 a_{ij}\langle \vec{e_k},\vec{e_j}\rangle-\frac 1 2 a_{ij}a_{kj}\langle\vec x,\vec n\rangle
\end{align}

Setting $k=i$ and summing over $i$, we obtain $$\Delta H =-\frac 1 2 a_{ij,i}\langle\vec x,\vec{e_j}\rangle-\frac 1 2 a_{ii}-\frac 1 2 |A|^2\langle\vec x,\vec n\rangle.$$  By the Codazzi Equation, we get that $a_{ij,i}=a_{ii,j}$.  Combining this with the fact that $H=-a_{ii}$ yields
\begin{align} \Delta H \notag &= \frac 1 2 \langle\vec x,\nabla H\rangle + \frac 1 2 H -|A|^2 H \\
L H \notag &= \Delta H -\frac 1 2 \langle \vec x,\nabla H\rangle +(\frac 1 2 +|A|^2)H \\
L H \notag &= \frac 1 2 \langle\vec x,\nabla H\rangle + \frac 1 2 H -|A|^2 H -\frac 1 2 \langle \vec x,\nabla H\rangle +(\frac 1 2 +|A|^2)H \\
L H \notag &= H
\end{align}

We now consider a constant vector field $\vec v$.  Let $f=\langle\vec v,\vec n\rangle$, and we compute.
\begin{align} \nabla_{\vec{e_i}} f \notag &= \langle\vec v,\nabla_{\vec{e_i}}\vec n\rangle \\
\nabla_{\vec{e_i}} f \notag &= -a_{ij}\langle\vec v,\vec{e_j}\rangle \\
\nabla_{\vec{e_k}}\nabla_{\vec{e_i}}f \notag &= -a_{ij,k}\langle\vec v,\vec{e_j}\rangle -a_{ij}a_{kj}\langle\vec v,\vec n\rangle \\
\Delta f \notag &= \langle \vec v,\nabla H\rangle - |A|^2 f
\end{align}

From above, we know that
\begin{align} \nabla_{\vec{e_i}}H \notag &= -\frac 1 2 a_{ij}\langle\vec x,\vec{e_j}\rangle \\
\nabla H \notag &= -\frac 1 2 a_{ij}\langle\vec x,\vec{e_j}\rangle\vec{e_i} \\
\langle\vec v,\nabla H\rangle \notag &= -\frac 1 2 a_{ij}\langle\vec x,\vec{e_j}\rangle\langle\vec v,\vec{e_i}\rangle \\
\langle\vec v,\nabla H\rangle \notag &= \frac 1 2 \langle\vec x,-a_{ij}\langle\vec v,\vec{e_i}\rangle\vec{e_j}\rangle \\
\langle\vec v,\nabla H\rangle \notag &= \frac 1 2 \langle\vec x,\nabla f\rangle \\
Lf \notag &= \Delta f - \frac 1 2 \langle\vec x,\nabla f\rangle +(|A|^2+\frac 1 2)f \\
Lf \notag &= \frac 1 2 \langle\vec x,\nabla f\rangle -|A|^2 f - \frac 1 2 \langle\vec x,\nabla f\rangle +(|A|^2+\frac 1 2)f \\
Lf \notag &= \frac 1 2 f
\end{align}
\end{proof}

In \cite{KKM} it is noted that on the plane $\R^2\subset\R^3$, when conjugated by a Gaussian $L$ is equal to the Hamiltonian operator for the two-dimensional quantum harmonic oscillator plus a constant.  Since the quantum harmonic oscillator is well understood, this technique helps us understand the operator $L$.  In the following we adapt this idea to a much larger class of self-shrinking hypersurfaces of $R^{n+1}$.  With this in mind, we now discuss the well known eigenvalues and eigenfunctions of the quantum harmonic oscillator Hamiltonian.

\begin{defn}[Hermite Polynomials]\label{44} \cite{T}
Consider the operator on $\R$ given by $$E=-\left(\frac 1{2m}\right)\frac{d^2}{dx^2}+\frac 1 2 m \omega^2 x^2.$$  In quantum mechanics, the first term is the kinetic energy operator (using the convention that $\hbar =1$), and the second term denotes the potential energy of the one dimensional harmonic oscillator.  The sum of these is the Hamiltonian operator, which gives the total energy of the harmonic oscillator.  The possible energy levels of the oscillator are then exactly the eigenvalues of the operator $E$, and the eigenfunctions are the states corresponding to these energy levels.
Due to its importance in quantum mechanics, this problem has been extensively studied.

We define the $k$th Hermite polynomial to be $$H_k(z)=(-1)^k e^{z^2}\frac {d^k}{dz^k}e^{-z^2}.$$  Then the eigenvalues of the Hamiltonian $E$ are given by $$\left\{-\omega\left(k+\frac 1 2\right):k=0,1,2,3,...\right\}$$ and the corresponding eigenfunctions are $$\psi_k=\sqrt{\frac 1 {2^k(k!)}}\left(\frac{m\omega}{\pi}\right)^{\frac 1 4} e^{-\frac{m\omega x^2}2}H_k(\sqrt{m\omega}x).$$
These eigenfunctions form an orthonormal basis of $L^2(\R)$.  Note that these eigenvalues differ from the quantum mechanical energy levels by a factor of $-1$ due to our choice of sign convention in Definition \ref{10}.

For future reference, we record the first few Hermite polynomials.  In our applications below, we will be interested in the case when $m=\omega=\frac 1 2$, so we will record the polynomials $H_k(\sqrt{m\omega}x)=H_k(\frac x 2).$
\begin{align} H_0\left(\frac x 2\right) \notag &= 1 \\
H_1\left(\frac x 2\right) \notag &= x \\
H_2\left(\frac x 2\right) \notag &= x^2-2 \\
H_3\left(\frac x 2\right) \notag &= x^3-6x
\end{align}

\end{defn}

We now prove a very helpful lemma.

\begin{lem}\label{41} Suppose $\Sigma^{n-1}\subset\R^n$ is a smooth complete embedded self-shrinker without boundary and with polynomial volume growth.  Suppose $\{g_m\}_{m=0}^{\infty}$ is an orthonormal basis of the weighted space $L^2(\Sigma)$ made up of eigenfunctions of $L$ satisfying $$L g_m=-\lambda_m g_m.$$  Then $\Sigma^{n-1}\times\R\subset\R^{n+1}$ is also a self-shrinker. Furthermore, the eigenvalues of $L$ on $\Sigma^{n-1}\times\R\subset\R^{n+1}$ are
$$\left\{\lambda_m+\frac 1 2 k:m,k=0,1,2,3,...\right\},$$ and the corresponding eigenfunctions are $g_m H_k(\frac {x_n}2).$  Here $x_n$ is the coordinate function on the $\R$ factor of $\Sigma^{n-1}\times\R$, and $H_k$ denotes the $k$th Hermite polynomial.  We have omitted the conventional normalizations for simplicity.
\end{lem}
\begin{proof} Fix a point $\vec p\in\Sigma^{n-1}$, and choose an orthonormal frame $\{\vec{e_i}\}_{i=1}^{n-1}$ for $\Sigma$ such that $\nabla_{\vec{e_i}}^{\top}\vec{e_j}=0$.  Then letting $\vec{e_n}$ be a unit vector tangent to the $\R$ factor of $\Sigma\times\R$ and $x_n$ be the coordinate on this factor, we obtain an orthonormal frame $\{\vec{e_i}\}_{i=1}^n$ for $\Sigma\times\R$.  Note that the mean curvature of $\Sigma^{n-1}$ equals that of $\Sigma^{n-1}\times\R$.  Also, $\langle\vec{e_n},\vec n\rangle=0$, so Lemma \ref{30} shows that $\Sigma^{n-1}\times\R\subset\R^{n+1}$ is a self-shrinker.

For the rest of this proof, we will use the subscript $_\Sigma$ to denote operators and quantities restricted to the surface $\Sigma^{n-1}\subset\R^n$.  Operators and quantities without the subscript $_\Sigma$ will refer to the surface $\Sigma^{n-1}\times\R\subset\R^{n+1}$.

Since $\R$ is a straight line, $\vec{e_n}$ is a constant vector field.  Thus $\nabla_{\vec{e_i}}\vec{e_n}=0$ for all $i$, so $|A_\Sigma|^2=|A|^2$.  Likewise, the Laplacian on $\Sigma\times\R$ splits as
\begin{align} \Delta \notag &=\Delta_\Sigma + (\nabla_{\vec{e_n}})^2 \\
L_{\Sigma}f \notag &= \Delta_{\Sigma}f-\frac 1 2 \langle\vec x,\nabla_\Sigma f\rangle + \left(|A_\Sigma|^2+\frac 1 2\right)f \\
L_\Sigma f \notag &= e^{\frac{|\vec {x}_\Sigma|^2}8} (\hat H_\Sigma) e^{-\frac{|\vec {x}_\Sigma|^2}8} f \\
L_\Sigma f \notag &= e^{\frac{|\vec {x}|^2}8} (\hat H_\Sigma) e^{-\frac{|\vec {x}|^2}8} f
\end{align}

Here we are defining $\hat H_{\Sigma}$ to be $$\hat H_\Sigma=e^{-\frac{|\vec {x}_\Sigma|^2}8} L_\Sigma e^{\frac{|\vec {x}_\Sigma|^2}8}.$$  We now show that the eigenvalues of $L_\Sigma$ and $\hat H_\Sigma$ are identical, and their eigenvectors differ by a factor of $e^{-\frac{|\vec {x}_\Sigma|^2}8}$.  Suppose $g_m$ is an eigenfunction of $L_\Sigma$ with eigenvalue $\lambda_m$.  This means that $$ L_\Sigma g_m = -\lambda_m g_m $$ and $g_m$ is in $L^2(\Sigma)$ with the weighted measure $e^{-\frac{|\vec x_\Sigma|^2}4}d\mu$.  In other words $$\int\limits_{\Sigma}|g_m |^2 \left(e^{-\frac{|\vec {x}_\Sigma|^2}4}d\mu\right) <\infty.$$ We now show that $\lambda_m$ is also an eigenvalue of $\hat H_\Sigma$ with corresponding eigenfunction $e^{-\frac{|\vec x_\Sigma|^2}{8}} g_m.$

\begin{align} \hat H_\Sigma \notag &=e^{-\frac{|\vec {x}_\Sigma|^2}8} L_\Sigma e^{\frac{|\vec {x}_\Sigma|^2}8} \\
\hat H_\Sigma e^{-\frac{|\vec x_\Sigma|^2}{8}} g_m \notag &=e^{-\frac{|\vec {x}_\Sigma|^2}8} L_\Sigma e^{\frac{|\vec {x}_\Sigma|^2}8} e^{-\frac{|\vec x_\Sigma|^2}{8}} g_m \\
\hat H_\Sigma e^{-\frac{|\vec x_\Sigma|^2}{8}} g_m \notag &=e^{-\frac{|\vec {x}_\Sigma|^2}8} L_\Sigma g_m \\
\hat H_\Sigma e^{-\frac{|\vec x_\Sigma|^2}{8}} g_m \notag &=e^{-\frac{|\vec {x}_\Sigma|^2}8} (-\lambda_m g_m )\\
\hat H_\Sigma e^{-\frac{|\vec x_\Sigma|^2}{8}} g_m \notag &= -\lambda_m e^{-\frac{|\vec {x}_\Sigma|^2}8}  g_m
\end{align}

We also check the integrability condition.  In order to be an eigenfunction for $\hat H_\Sigma$, we must show that $e^{-\frac{|\vec x_\Sigma|^2}{8}} g_m\in L^2(\Sigma)$ with respect to the standard measure $d\mu$.  However, this is trivial since
$$\int\limits_{\Sigma}|e^{-\frac{|\vec x_\Sigma|^2}{8}}g_m |^2 d\mu = \int\limits_{\Sigma}|g_m |^2 e^{-\frac{|\vec {x}_\Sigma|^2}4}d\mu <\infty.$$  These arguments work in both directions, so there is a one to one correspondence between the eigenvalue-eigenfunction pairs of $\hat H_\Sigma$ and those of $L_\Sigma$.

We now turn our attention to the operator $L$ on $\Sigma\times\R$.

\begin{align} Lf \notag &= \Delta f - \frac 1 2 \langle \vec x,\nabla f\rangle+\left(|A|^2+\frac 1 2\right)f \\
Lf \notag &= \Delta_\Sigma f +(\nabla_{\vec{e_n}})^2 f - \frac 1 2 \langle \vec x,\nabla_\Sigma f\rangle-\frac 1 2 \langle \vec x,\vec{e_n}\nabla_{\vec{e_n}} f\rangle  +\left(|A_\Sigma|^2+\frac 1 2\right)f \\
Lf \notag &= L_\Sigma f + (\nabla_{\vec{e_n}})^2 f -\frac 1 2 \langle \vec x,\vec{e_n}\nabla_{\vec{e_n}} f\rangle \\
\nabla_{\vec{e_n}}\left( e^{-\frac{|\vec {x}|^2}8}f \right) \notag &= -\frac 1 4 \langle \vec{e_n},\vec x\rangle e^{-\frac{|\vec {x}|^2}8}f + e^{-\frac{|\vec {x}|^2}8}\nabla_{\vec{e_n}}f \\
(\nabla_{\vec{e_n}})^2\left( e^{-\frac{|\vec {x}|^2}8}f \right) \notag &= -\frac 1 4 \langle \vec{e_n},\vec {e_n}\rangle e^{-\frac{|\vec {x}|^2}8}f - \frac 1 4 \langle \vec{e_n},\vec x\rangle e^{-\frac{|\vec {x}|^2}8}\nabla_{\vec{e_n}} f +\frac 1{16}\langle\vec{e_n},\vec x\rangle^2 e^{-\frac{|\vec {x}|^2}8} f \\
\notag & \ \ \ -\frac 1 4 \langle \vec{e_n},\vec x\rangle e^{-\frac{|\vec {x}|^2}8}\nabla_{\vec{e_n}} f + e^{-\frac{|\vec {x}|^2}8}(\nabla_{\vec{e_n}})^2 f \\
(\nabla_{\vec{e_n}})^2\left( e^{-\frac{|\vec {x}|^2}8}f \right) \notag &= \left(-\frac 1 4 f -\frac 1 2 \langle \vec x,\vec{e_n}\nabla_{\vec{e_n}} f\rangle +\frac {x_n^2}{16} +(\nabla_{\vec{e_n}})^2 f  \right)e^{-\frac{|\vec {x}|^2}8} \\
Lf \notag &= e^{\frac{|\vec {x}|^2}8} \left( \hat H_\Sigma +(\nabla_{\vec{e_n}})^2 - \frac{x_n^2}{16} +\frac 1 4\right) e^{-\frac{|\vec {x}|^2}8} f
\end{align}

Define the operator in parentheses to be $$\hat H = \hat H_\Sigma +(\nabla_{\vec{e_n}})^2 - \frac{x_n^2}{16} +\frac 1 4.$$  By the same argument as above, the eigenvalues of $L$ and $\hat H$ are identical, and their eigenvectors differ by a factor of $e^{-\frac{|\vec {x}|^2}8}$.  We define $\hat H_n = (\nabla_{\vec{e_n}})^2 - \frac{x_n^2}{16} +\frac 1 4$, so now
$$\hat H = \hat H_\Sigma + \hat H_n$$ and $\hat H_n$ depends only on $\vec{x_n}$.

This is useful, because $\hat H_n -\frac 1 4$ is the Hamiltonian of the quantum harmonic oscillator, which we discussed in Definition \ref{44}.  We wish to prove that all the eigenvalues and eigenvectors of $\hat H$ can be built out of the eigenvalues and eigenvectors of its two composite operators.

To this end, suppose that $\{g_m(\vec x_\Sigma)\}_{m=0}^{\infty}$ is an orthonormal basis of $L^2(\Sigma)$ (with the weighted metric) made up of eigenfunctions of $L_\Sigma$ with $$L_\Sigma g_m=-\lambda_m g_m.$$  Then $\{e^{-\frac{|\vec {x}_\Sigma|^2}8}g_m(\vec x_\Sigma)\}_{m=0}^{\infty}$ is an orthonormal basis of $L^2(\Sigma)$ (with the standard metric) made up of eigenfunctions of $\hat H_\Sigma$.  Suppose also that $\{e^{-\frac{x_n^2}8}h_k(x_n)\}_{k=0}^{\infty}$ is an orthonormal basis of $L^2(\R)$ made up of eigenfunctions of $\hat H_n$ such that $$\hat H_n e^{-\frac{x_n^2}8}h_k=-\mu_k e^{-\frac{x_n^2}8}h_k.$$
Then we claim that all eigenfunctions of $\hat H$ can be constructed by taking products
$$\left\{\left( e^{-\frac{|\vec {x}_\Sigma|^2}8} g_m(\vec x_\Sigma)\right)\left(e^{-\frac{x_n^2}8}h_k(x_n)\right):m,k=0,1,2, ...\right\}$$ which can be simplified to the following.
$$\{e^{-\frac{|\vec {x}|^2}8} g_m(\vec x_\Sigma) h_k(x_n):m,k=0,1,2, ...\}$$

Furthermore, these eigenfunctions satisfy
$$\hat H  e^{-\frac{|\vec {x}|^2}8} g_m h_k =-(\lambda_m + \mu_k)e^{-\frac{|\vec {x}|^2}8}g_m h_k.$$

We check the last claim first.  Fix a specific $m$ and $k$ and compute.
\begin{align} \hat H  e^{-\frac{|\vec {x}|^2}8} g_m h_k  \notag &= (\hat H_\Sigma +\hat H_n) e^{-\frac{|\vec {x}|^2}8} g_m h_k \\
\hat H  e^{-\frac{|\vec {x}|^2}8} g_m h_k  \notag &= \left(\hat H_\Sigma e^{-\frac{|\vec {x}_\Sigma|^2}8}g_m\right)e^{-\frac{|x_n|^2}8}h_k +e^{-\frac{|\vec {x}_\Sigma|^2}8}g_m\left(\hat H_n e^{-\frac{|x_n|^2}8}h_k \right) \\
\hat H  e^{-\frac{|\vec {x}|^2}8} g_m h_k  \notag &= \left(-\lambda_m e^{-\frac{|\vec {x}_\Sigma|^2}8}g_m\right)e^{-\frac{|x_n|^2}8}h_k +e^{-\frac{|\vec {x}_\Sigma|^2}8}g_m\left(-\mu_k e^{-\frac{|x_n|^2}8}h_k \right) \\
\hat H  e^{-\frac{|\vec {x}|^2}8} g_m h_k  \notag &= -(\lambda_m + \mu_k)e^{-\frac{|\vec {x}|^2}8}g_m h_k
\end{align}

Likewise, the integrability condition is clearly satisfied.  This shows that products of the form $$\{g_m(\vec x_\Sigma) h_k(x_n):m,k=0,1,2, ...\}$$ are eigenfunctions of the operator $L$ on $\Sigma\times\R$.  Thus, it suffices to show that these eigenvectors form a basis of $L^2(\Sigma\times\R)$.  For this, it suffices to show that they are complete.

To this end, let $f(\vec x)\in L^2(\Sigma\times\R)$ be arbitrary.  Then define
\begin{align} b_k(\vec x_\Sigma) \notag &=\int\limits_{\R} f(\vec x)h_k(x_n)\left(e^{-\frac{|x_n|^2}4}dx_n\right) \\
|b_k(\vec x_\Sigma)|^2 \notag &\leq \int\limits_{\R} |f(\vec x)|^2 \left(e^{-\frac{|x_n|^2}4}dx_n\right) \int\limits_{\R} |h_k(x_n)|^2 \left(e^{-\frac{|x_n|^2}4}dx_n\right) \\
|b_k(\vec x_\Sigma)|^2 \notag &\leq \int\limits_{\R} |f(\vec x)|^2 \left(e^{-\frac{|x_n|^2}4}dx_n\right) \\
\int\limits_{\Sigma}|b_{k}(\vec x_\Sigma)|^2 \left(e^{-\frac{|\vec {x}_\Sigma|^2}4}d\vec x_\Sigma\right) \notag &\leq \int\limits_{\Sigma}\int\limits_{\R} |f(\vec x)|^2 \left(e^{-\frac{|x_n|^2}4}dx_n\right) \left(e^{-\frac{|\vec {x}_\Sigma|^2}4}d\vec x_\Sigma\right) \\
\int\limits_{\Sigma}|b_{k}(\vec x_\Sigma)|^2 \left(e^{-\frac{|\vec {x}_\Sigma|^2}4}d\vec x_\Sigma\right) \notag &\leq \| f(\vec x)\|_{L^2(\Sigma\times\R)}
\end{align}

Thus, $b_{k}(\vec x_\Sigma)\in L^2(\Sigma)$, so we can write $b_{k}(\vec x_\Sigma)=\sum\limits_{m=0}^\infty b_{mk}g_{m}(\vec x_\Sigma)$.  Thus, $f(\vec x)=\sum\limits_{m,k=0}^\infty b_{mk}g_{m}h_{k}$.  Thus we have an orthonormal basis of $L^2(\Sigma\times\R)$.

We now turn our attention to the operator $\hat H_n -\frac 1 4$.  From Definition \ref{44}, this is the negative of the Hamiltonian of the quantum harmonic oscillator with $m=\frac 1 2$ and $\omega=\frac 1 2$.  The eigenvalues of this operator are $\{\frac 1 4 +\frac 1 2 k:k=0,1,2,3, ...\}$ with corresponding eigenfunctions $h_k=e^{-\frac{x_n^2}{8}}H_k(\frac {x_n} 2)$, where $H_k$ is the $k$th Hermite polynomial.  Thus the eigenfunctions of $\hat H_n$ are the same, but the eigenvalues of $\hat H_n$ are all lowered by $\frac 1 4$ to become $$\left\{\frac 1 2 k:k=0,1,2,3, ...\right\}.$$

By assumption, the eigenvalues of $\hat H_\Sigma$ are $\{\lambda_m\}$ with eigenfunctions $\left\{e^{-\frac{|\vec x_\Sigma|^2}8}g_m\right\}$.  Thus the eigenvalues of $L$ are $$\left\{\lambda_m+\frac 1 2 k:m,k=0,1,2,3,...\right\},$$ and the corresponding eigenfunctions are $\left\{g_m(\vec x_\Sigma) H_k(\frac {x_n}2)\right\}.$

\end{proof}

\section{Proof of Low Index Classification}

In this section, we will apply Lemma \ref{41} in order to prove the classification result, Theorem \ref{37}.

\begin{Pro}\label{39} Let $\Sigma^n\subset\R^{n+1}$ be a hyperplane through the origin.  Then the eigenvalues of $L$ on $\Sigma$ are
$$\left\{-\frac 1 2 + \frac 1 2 \sum\limits_{i=1}^n k_i:k_i=0,1,2,3,... \right\}.$$
For each choice of the $k_i$'s there exists a unique eigenfunction given by $$\prod\limits_{i=1}^n H_{k_i}\left(\frac {x_i}2\right).$$  Thus in particular, the index of $L$ on $\Sigma$ is $1$.
\end{Pro}

\begin{proof} We first restrict attention to $\Sigma^1=\R\subset\R^2$.  Without loss of generality assume $\Sigma$ is the $x$-axis.  Then
\begin{align}  Lf \notag &= \Delta f - \frac 1 2 \langle\vec x,\nabla f\rangle +(|A|^2+\frac 1 2)f \\
|A|^2 \notag &= 0 \\
Lf \notag &= \partial_x^2 f-\frac 1 2 x\partial_x f +\frac 1 2 f \\
\partial_x \left(e^{-\frac{x^2}8} f \right) \notag &= -\frac x 4 e^{-\frac{x^2}8} f +e^{-\frac{x^2}8} \partial_x f \\
\partial_x^2 \left(e^{-\frac{x^2}8} f \right) \notag &= \left(-\frac 1 4 f + \frac{x^2}{16}f -\frac{2x}4 \partial_x f +\partial_x^2 f \right) e^{-\frac{x^2}8} \\
e^{\frac{x^2}8} \partial_x^2 \left(e^{-\frac{x^2}8} f \right) \notag &= -\frac 1 4 f + \frac{x^2}{16}f -\frac{x}2 \partial_x f +\partial_x^2 f \\
Lf \notag &= e^{\frac{x^2}8}\left( \partial_x^2 -\frac{x^2}{16} +\frac 3 4 \right)e^{-\frac{x^2}8} f
\end{align}

We define the operator in parentheses above as $$\hat{H}=  \partial_x^2 -\frac{x^2}{16} +\frac 3 4.$$  We analyzed this situation in the proof of Lemma \ref{41}.  We showed there that on $\Sigma$, $\lambda$ is an eigenvalue of $\hat{H}$ with eigenfunction $u_\lambda$ if and only if $\lambda$ is also an eigenvalue of $L$ with corresponding eigenfunction $e^{\frac{x^2}{8}}u_\lambda.$

As in the same proof, $\hat H-\frac 3 4$ is the negative of the quantum harmonic oscillator Hamiltonian with eigenvalues $$\left\{\frac 1 4 +\frac 1 2 k: k=0,1,2,...\right\}.$$  Thus, the eigenvalues of $L$ on $\R\subset\R^2$ are $$\left\{-\frac 1 2 +\frac 1 2 k: k=0,1,2,...\right\}$$ with corresponding eigenfunctions given by $$\left\{H_k(\frac x 2)\right\}$$  where $H_k(z)$ denotes the $k$th Hermite polynomial from Definition \ref{44}.

We then extend this to a general dimensional hyperplane $\Sigma^n\subset \R^{n+1}$ by successive applications of Lemma \ref{41}.  A simple induction argument shows that the eigenvalues of $L$ on $\Sigma^n\subset\R^{n+1}$ are $$\left\{-\frac 1 2 + \frac 1 2 \sum\limits_{i=1}^n k_i:k_i=0,1,2,3,... \right\}$$ with corresponding eigenfunctions $$\left\{\prod\limits_{i=1}^n H_{k_i}(\frac {x_i}2):k_i=0,1,2,...\right\}.$$

Thus, the index of $L$ on a hyperplane $\Sigma$ is $1$.  The only negative eigenvalue is $-\frac 1 2$, and the corresponding eigenfunction is the constant function $f=1$.
\end{proof}

\begin{rem} Note that the negative eigenvalue found in Proposition \ref{39} is exactly the one we knew existed from Theorem \ref{38}.  On a hyperplane, the mean curvature $H$ is identically $0$.  Also, if $\vec v$ is any constant vector field tangent to $\Sigma$, then $\langle \vec v,\vec n\rangle\equiv 0$.  Taking $\vec v = \vec n$ gives the eigenfunction $\langle \vec v,\vec n\rangle =1$ with eigenvalue $-\frac 1 2$.
\end{rem}

\begin{Pro}\label{40} Let $\Sigma^n=\mathbb S^k\times\R^{n-k}\subset \R^{n+1}$ be a self-shrinker with $1\leq k\leq n$.  Then the eigenvalues of $L$ on $\Sigma$ are
$$\left\{-1+\frac 1{2k}m(m+k-1)+\frac 1 2 \sum\limits_{i=k+1}^n c_i:m,c_i=0,1,2,3, ... \right\}.$$

For a fixed choice of the $c_i$'s and a fixed $m$, the number of independent eigenfunctions is given by the number of independent harmonic homogeneous polynomials of degree $m$ in $k+1$ variables.  In particular, the index of $L$ on $\Sigma$ is $n+2$.
\end{Pro}
\begin{proof} Note that in order to be a self-shrinker the $\mathbb S^k$ factor of $\Sigma$ must have radius $\sqrt{2k}$.  We first restrict attention to the case $\Sigma=\mathbb S^k\subset \R^{k+1}$.  Note that on $\mathbb S^k$ all tangent vectors are perpendicular to $\vec x$.  Thus
\begin{align} |A|^2 \notag &= \sum\limits_{i=1}^k \frac 1 {2k} \\
|A|^2 \notag &= \frac 1 2 \\
Lf \notag &= \Delta f -\frac 1 2 \langle\vec x,\nabla f\rangle +f \\
Lf \notag &= \Delta f +f
\end{align}

Switching to spherical coordinates gives $$L=\frac 1{2k}\Delta_{\mathbb S^k} +1.$$  However, it is well known (see \cite{T}) that the eigenvalues of
$\Delta_{\mathbb S^k}$ are $$\{m(m+k-1):m=0,1,2,3,...\}$$
with corresponding eigenfunctions given by the spherical harmonics.  The multiplicity of each eigenvalue is given by the number of harmonic homogeneous polynomials of degree $m$ in $k+1$ variables.

Thus the eigenvalues of $L$ on $\Sigma^k=\mathbb S^k\subset \R^{k+1}$ are $$\left\{-1+\frac 1{2k}m(m+k-1):m=0,1,2,3, ... \right\}.$$  We now extend this to the case where $\Sigma^n=\mathbb S^k\times\R^{n-k}\subset \R^{n+1}$ via $n-k$ applications of Lemma \ref{41}.  This shows that the eigenvalues of $L$ on $\Sigma^n$ are
$$\left\{-1+\frac 1{2k}m(m+k-1)+\frac 1 2 \sum\limits_{i=k+1}^n c_i:m,c_i=0,1,2,3, ... \right\}$$ with corresponding eigenfunctions given by spherical harmonics multiplied by Hermite polynomials.

When $m=c_i=0$, we obtain the lowest eigenvalue $-1$, corresponding to the constant eigenfunction given by the mean curvature of $\Sigma$.  When all the $c_i=0$ and $m=1$, we obtain the eigenvalue $-\frac 1 2$ with multiplicity $k+1$.  The $k+1$ eigenfunctions are the restrictions to $\Sigma$ of the homogeneous linear polynomials in $k+1$ variables.  However, the eigenvalue $-\frac 1 2$ has additional eigenfunctions coming from the case when $m=0$ and exactly one $c_i=1$ while the others are $0$.  There are $n-k$ choices of a $c_i$, so the total multiplicity of $-\frac 1 2 $ is $(n-k)+(k+1)=n+1$.

It is interesting to note that the $n+1$ eigenfunctions for eigenvalue $-\frac 1 2$ are given by the restrictions of the $n+1$ Euclidean coordinate functions to $\Sigma$.
\end{proof}

We have now proven the first two claims of Theorem \ref{37}.  In order to complete the proof of this theorem, we will need the following facts from \cite{CM1}.

\begin{thm}\label{33} $\mathbb S^k\times\R^{n-k}$ are the only smooth complete embedded self-shrinkers without boundary, with polynomial volume growth, and $H\geq 0$ in $\R^{n+1}$.
\end{thm}

\begin{thm}\label{42} Any smooth complete embedded self-shrinker in $\R^3$ without boundary and with polynomial area growth that splits off a line must either be a plane or a round cylinder.
\end{thm}

\begin{thm}\label{43} If the mean curvature $H$ changes sign, then the first eigenvalue of $L$ is strictly less than $-1$.
\end{thm}

We are now ready to prove Theorem \ref{37}.

\begin{proof}[proof of Theorem \ref{37}]  The cases when $\Sigma=\mathbb S^k\times\R^{n-k}\subset \R^{n+1}$ are covered in Propositions \ref{39} and \ref{40}.  We therefore assume for the rest of the proof that $\Sigma^n\neq \mathbb S^k\times\R^{n-k}$ for any $k$.

The proof procedes by induction on the dimension $n$ of $\Sigma^n$.  We begin with the base case $n=2$.  By Theorem \ref{33}, the mean curvature $H$ of $\Sigma$ changes sign.  Thus Theorem \ref{43} gives that the first eigenvalue $\mu_1$ of $L$ satisfies $\mu_1<-1$.  However, by Theorem \ref{38} we know $LH=H$, so since $H$ is not identically $0$ we know that $-1$ is also an eigenvalue of $L$.

We are assuming $\Sigma^2\neq \mathbb S^k\times\R^{2-k}$, so Theorem \ref{42} gives that $\Sigma$ does not split off a line.  This means that there is no nonzero constant vector field $\vec v$ such that $\langle \vec v,\vec n\rangle\equiv 0$.  Thus Theorem \ref{38} gives that $-\frac 1 2$ is an eigenvalue of $L$ with multiplicity at least $3$.  Thus, the index of $\Sigma$ is at least $n+3=5$.  It is not possible that $\Sigma^2\neq \mathbb S^k\times\R^{2-k}$ and also splits off a line, so the final claim is trivially true.

Now assume that the theorem holds for all surfaces $\tilde\Sigma^{n-1}\subset\R^n$, and consider an arbitrary self-shrinker $\Sigma^n\subset\R^{n+1}$ such that $\Sigma^n\neq \mathbb S^k\times\R^{n-k}$. We have two cases.  Either $\Sigma^n$ splits off a line, or it does not.

Suppose $\Sigma^n$ does not split off a line.  In this case there is no nonzero constant vector field $\vec v$ such that $\langle \vec v,\vec n\rangle\equiv 0$.  Thus Theorem \ref{38} gives that $-\frac 1 2$ is an eigenvalue of $L$ with multiplicity at least $n+1$.  However, we also know from Theorem \ref{33} that the mean curvature $H$ changes sign.  Thus, $H$ is not identically $0$, so $-1$ is an eigenvalue of $L$.  Also, Theorem \ref{43} gives the existence of at least one eigenvalue lower than $-1$.  Thus, the index of $\Sigma$ is at least $n+3$.

We now consider the other case, so suppose $\Sigma^n$ splits off a line.  In this case, there exists some $\tilde\Sigma^{n-1}\subset\R^n$ such that $\Sigma=\tilde\Sigma\times\R$.  Note that $\tilde\Sigma\neq\mathbb S^k\times\R^{n-1-k}$, since that would contradict our assumption that $\Sigma^n\neq \mathbb S^k\times\R^{n-k}$.  By our inductive hypothesis the index of $\tilde\Sigma^{n-1}$ is at least $n+2$.  We also know that since $H$ changes sign, one of the eigenvalues of $L$ on $\tilde\Sigma^{n-1}$ is $-1$ and another eigenvalue $\mu_1<-1$.  We now apply Lemma \ref{41} to see that $\Sigma$ has the same negative eigenvalues as $\tilde\Sigma$ as well as at least two new ones, namely $-1+\frac 1 2$ and $\mu_1+\frac 1 2$.  Thus, in this case the index of $\Sigma$ is at least $n+4$, which completes the proof.
\end{proof}

\newpage

\chapter{Stability of Pieces of $\mathbb S^k\times\R^{n-k}$}\label{24}

\section{Eigenfunctions with Eigenvalue $0$}

In the previous chapter we showed that all hypersurfaces of the form $\mathbb S^k\times\R^{n-k}\subset\R^{n+1}$ have finite index.  Then by Theorem \ref{45} each of these surfaces must be stable outside of some compact set.  In this chapter, we look for the largest stable subsets of these self-shrinkers.  For the rest of this chapter, we will let $\vec x=(x_{k+1},x_{k+2},...,x_n)$ denote Euclidean coordinates on the $\R^{n-k}$ factor of $\Sigma$.  We will also let $\vec\phi$ denote spherical coordinates on the $\s^k$ factor of $\Sigma.$  Then $(\vec\phi,\vec x)$ are coordinates on $\Sigma$.

We will proceed by finding positive Jacobi functions as in Definition \ref{8}.  Note that eigenfunctions with eigenvalue $0$ are Jacobi functions.  This allows us to exploit the results of Chapter \ref{26} to easily prove the following series of propositions giving stable subsets of $\s^k\times\R^{n-k}$.

\begin{Pro} Let $\Sigma^n\subset\R^{n+1}$ be a flat hyperplane through the origin.  Suppose $P\subset\Sigma$ is a flat $(n-1)$-plane through the origin.  Then $P$ splits $\Sigma^n$ into two stable half-hyperplanes.
\end{Pro}
\begin{proof} By Proposition \ref{39} any coordinate function $x_i$ on $\Sigma$ satisfies $L x_i=0$.  Thus, $x_i$ is a Jacobi function on the portion of $\Sigma$ given by $\{x_i>0\}$.  The result follows by changing coordinates.
\end{proof}

In the previous proposition, $\Sigma=\s^k\times\R^{n-k}$ for the value $k=0$.  For comparison, and because we will use it later, we record the following fact.

\begin{Pro}\label{9} Let $\Sigma=\s^k\times\R^{n-k}\subset\R^{n+1}$ be a self-shrinker with $1\leq k\leq n-1$.  Let $H\subset\R^{n-k}$ be a half-space with boundary equal to a flat $(n-k-1)$-plane through the origin.  Then $\s^k\times H\subset \Sigma$ is unstable.
\end{Pro}
\begin{proof} It suffices to show the existence on $\s^k\times H$ of an eigenfunction with negative eigenvalue and $0$ boundary value.  However, from \ref{40} we know that $x_n$ is an eigenfunction of $L$ with eigenvalue $-\frac 1 2$.  By changing coordinates, we can make the half-space $H$ equal to the region $\{x_n>0\}.$
\end{proof}

\begin{Pro} Let $\Sigma=\s^k\times\R^{n-k}\subset\R^{n+1}$ be a self-shrinker with $1\leq k\leq n-1$.  Let $C\subset\s^k$ be an arbitrary hemisphere, and $H\subset\R^{n-k}$ be a half-space with boundary equal to a flat $(n-k-1)$-plane through the origin.  Then $C\times H\subset \Sigma$ is stable.
\end{Pro}
\begin{proof} By Proposition \ref{40} we can obtain an eigenfunction with eigenvalue $0$ by taking $m=c_n=1$ and all other $c_i=0$.  The spherical harmonics corresponding to $m=1$ are obtained by considering $\s^k\subset\R^{k+1}$ and restricting the coordinate functions of $\R^{k+1}$ to the surface $\s^k$.  By choice of coordinates, we can thus obtain any hemisphere $C$ as the portion of $\s^k$ on which $g_c$ is positive, where $g_c$ is some spherical harmonic with $m=1.$  Likewise, we can choose coordinates on $\R^{n-k}$ such that $H$ is the half-space given by $x_n>0.$  Then the eigenfunction $x_n g_c$ satisfies $$L (x_n g_c)=0$$ and $x_n g_c>0$ on the set $C\times H\subset \Sigma$.  Thus $x_n g_c$ is a positive Jacobi function, so the result follows.
\end{proof}

\begin{Pro} Let $\Sigma=\s^k\times\R^{n-k}\subset\R^{n+1}$ be a self-shrinker with $1\leq k\leq n-2$.  Let $P_1,P_2\subset \R^{n-k}$ be flat, orthogonal $(n-k-1)$-planes through the origin.  Then $P_1\cup P_2$ splits $\R^{n-k}$ into four quarter-spaces $Q_i\subset \R^{n-k}$.  Each quarter-space $\s^k\times Q_i\subset\Sigma$ is stable.
\end{Pro}
\begin{proof} By Proposition \ref{40} we can obtain an eigenfunction with eigenvalue $0$ by taking $c_{n-1}=c_n=1$ and all other $m,c_i=0$.  In this case, the eigenfunction is $x_{n-1}x_n$, which is positive on $\{x_n>0\}\cap\{x_{n-1}>0\}$.  This set can be made equal to any $Q_i$ defined above by changing coordinates, so the result follows.
\end{proof}

\begin{Pro}\label{46} Let $\Sigma=\s^k\times\R^{n-k}\subset\R^{n+1}$ be a self-shrinker with $1\leq k\leq n-1$.  Then the following subsets of $\Sigma$ are stable.
\begin{enumerate}\item $\{x_n>\sqrt 2\}$
\item $\{|x_n|<\sqrt 2\}$
\item $\{x_n<-\sqrt 2\}$
\item $\{|\vec x|>\sqrt {2(n-k)}\}$
\item $\{|\vec x|<\sqrt {2(n-k)}\}$
\end{enumerate}
Note that by changing coordinates, the $x_n$-axis is arbitrary in $\R^{n-k}$.
\end{Pro}
\begin{proof} By Proposition \ref{40} we can obtain an eigenfunction with eigenvalue $0$ by taking $c_n=2$ and all other $m,c_i=0$.  In this case, the eigenfunction is $$x_n^2-2.$$  The first and third claims follow from finding the regions of $\Sigma$ on which this Jacobi function is positive.  Clearly $2-x_n^2$ is a positive Jacobi function on the region $\{|x_n|<\sqrt 2\}$, thus giving the second claim.

However, note that by choosing another $c_i=2$ while letting $c_n=0$, we obtain the eigenfunction $x_i^2-2.$  Then by linearity of the operator $L$ we have
\begin{align} L (x_i^2 -2) \notag &= 0 \\
L \left(\sum\limits_{i=n-k}^n x_i^2 -2\right) \notag &= 0 \\
L (|\vec x|^2 -2(n-k)) \notag &= 0
\end{align}
This Jacobi function is positive on the region $\{|\vec x|>\sqrt {2(n-k)}\}$, and its opposite is positive on the final region.  This completes the proof.
\end{proof}

\section{Rotationally Symmetric Stable Regions}

In Proposition \ref{46} we found some stable, radially symmetric portions of $\s^k\times\R^{n-k}$.  Regions of this form are given by $\{a <|\vec x|<b\}$.  In the following lemma we show that in order to find stable regions of this form, we need only consider radially symmetric Jacobi functions of the form $f=f(|\vec x|).$  This vastly simplifies the search for stable, radially symmetric regions by reducing a PDE to an ODE.

\begin{lem}\label{13} Suppose $f=f(\vec\phi,\vec x)$ is a positive Jacobi function on some portion $\{a<|\vec x|<b\}$ of $\mathbb S^k\times\R^{n-k}$.  Then the $\vec\phi$-rotationally symmetric function $g(\vec x)=\int\limits_{S^k} f(\vec\phi,\vec x)d\vec\phi$ is also a positive Jacobi function on the same region.

Suppose $g=g(\vec x)$ is a positive Jacobi function on some portion $\{a<|\vec x|<b\}$ of $\mathbb S^k\times\R^{n-k}$.  Then the function $h(r)=\int\limits_{|\vec x|=r} g(\vec x)$ is also a positive Jacobi function on the same region.
\end{lem}
\begin{proof} For both claims, positivity of the integral on the desired region follows trivially from positivity of the integrand.  It thus suffices to show in both cases that the differential operator $L$ yields $0$ when applied to the given integral.  However, this follows from the fact that in both cases the domain of integration is a compact set, so the operator $L$ commutes with integration.  Since in both cases the integrand is a Jacobi function, the result follows.
\end{proof}

Now we know we only need to consider symmetric Jacobi functions, so in the following proposition we apply the stability operator $L$ to radial functions.  We then find the general form a radial function must have in order to be a Jacobi function.

\begin{Pro}\label{47} Let $\Sigma=\s^k\times\R^{n-k}\subset\R^{n+1}$ be a self-shrinker with $0 \leq k\leq n-1$, and let $\vec x$ denote Euclidean coordinates on the $\R^{n-k}$ factor of $\Sigma$.  Suppose $f=f(r)$, where $r=|\vec x|$.  Let $L$ denote the stability operator.  Then there exist constants $K_1,K_2$ and $K_3$ such that the following hold.  In the following, $c_1$ and $c_2$ are arbitrary.

\begin{enumerate}
\item\label{64} If $k=n-1$, then $\Sigma=\s^{n-1}\times\R$, and $\vec x=x_n$.  Let $g=g(x_n)$.  Then $$Lg=g''-\frac {x_n}2 g' + g,$$ and the general solution to the differential equation $Lg=0$ is given by $$g(x_n)=c_1(x_n^2-2)+c_2 g_2(x_n)$$ where $g_2(x_n)$ is defined piecewise by

$$g_2(x_n) = \left\{
  \begin{array}{ll}
    K_1(x_n^2-2) + (x_n^2-2)\int\limits_2^{x_n} \frac{e^{\frac{z^2}4}}{(z^2-2)^2}dz & : x_n \in (\sqrt 2,\infty)\\
    -\frac 1 2\sqrt{\frac e 2} & : x = \sqrt 2 \\
    (x_n^2-2)\int\limits_0^{x_n} \frac{e^{\frac{z^2}4}}{(z^2-2)^2}dz & : x \in (-\sqrt 2,\sqrt 2)\\
    \frac 1 2\sqrt{\frac e 2} & : x =-\sqrt 2 \\
    -K_1(x_n^2-2) + (x_n^2-2)\int\limits_{-2}^{x_n} \frac{e^{\frac{z^2}4}}{(z^2-2)^2}dz & : x \in (-\infty,-\sqrt 2)\\
  \end{array}\right.$$

\item\label{65} If $1\leq k\leq n-2$, then $$Lf=f''+\left(\frac{n-k-1}r-\frac r 2\right)f'+f.$$  The dimensions only appear in the differential equation $Lf=0$ as a difference $n-k$, so the solutions likewise only depend on this difference.  We therefore set $\lambda=n-k$.  The general solution to the differential equation $Lf=0$ on $\{r\geq 0\}$ is given by $$f(r)=c_1(r^2-2\lambda)+c_2 f_2(r)$$ where $f_2(r)$ is defined on $(0,\infty)$ by

    $$f_2(r) = \left\{
      \begin{array}{ll}
         (r^2-2\lambda)\int\limits_1^r \frac{e^\frac {s^2}4}{s^{\lambda-1}(s^2-2\lambda)^2} ds & : r \in (0,\sqrt{2\lambda})\\
         -\frac 1 2 \left(\frac e {2\lambda}\right)^{\frac \lambda 2} & : r = \sqrt{2\lambda} \\
         (r^2-2\lambda)\left(K_2 + \int\limits_{2\sqrt{2\lambda}}^r \frac{e^\frac {s^2}4}{s^{\lambda-1}(s^2-2\lambda)^2} ds\right) & : r \in (\sqrt{2\lambda},\infty)\\
      \end{array}\right.$$
    Clearly this is only a global solution when $c_2=0$.

\item\label{66} If $k=0$, $\Sigma=\R^n$ is a flat plane, and $$Lf=f''+\left(\frac{n-1}r-\frac r 2\right) f'+\frac 1 2 f.$$  The general solution to the differential equation $Lf=0$ on $[0,\infty)$ is given by $$f(r)=c_1 f_1 +c_2 f_2$$ where $f_1$ and $f_2$ are defined below.  Note that $f_2$ is defined only on $\{r>0\}$, so we only obtain a global solution when $c_2=0.$

    $$f_1(r)=-1+\sum\limits_{m=1}^\infty \frac{m(2m-2)!}{2^{3m-1}(m!)^2\prod\limits_{j=0}^{m-1}(n+2j)}r^{2m}$$

    $$f_2(r) = \left\{
      \begin{array}{ll}
         f_1\int\limits_{\frac 1 2 r_1}^r \frac{e^{\frac{s^2}4}}{s^{n-1}f_1^2}ds & : r \in (0,r_1)\\
         -\frac{e^{\frac{r_1^2}4}}{r_1^{n-1} f_1 '(r_1)} & : r = r_1 \\
         K_3 f_1 + f_1\int\limits_{2r_1}^r \frac{e^{\frac{s^2}4}}{s^{n-1}f_1^2}ds & : r \in (r_1,\infty)\\
      \end{array}\right.$$

\end{enumerate}
\end{Pro}
\begin{proof} First recall from Equation \ref{3} we have that $$Lf=\Delta f -\frac 1 2 \langle (\vec\phi,\vec x), \nabla f\rangle+\left(|A|^2+\frac 1 2\right)f.$$  Here we are letting $\vec\phi$ denote coordinates on the $\s^k$ factor of $\Sigma$, so that $(\vec\phi,\vec x)$ are coordinates on $\Sigma$.  When $k=0$ we see $\Sigma=\R^n$ is a flat plane through the origin, so in particular $|A|^2=0$.  When $1\leq k\leq n-1,$ we have
\begin{align} |A|^2 \notag &= \sum\limits_{i=1}^k \frac 1 {2k} \\
|A|^2 \notag &= \frac 1 2
\end{align}

We note that the Laplacian on $\Sigma$ splits as a sum of Laplacians on its two factor spaces. $$\Delta_\Sigma=\Delta_{\s^k}+\Delta_{\R^{n-k}}$$  The functions we are considering here do not depend on the $\s^k$ factor, so that part of the Laplacian is $0$.  Recall that in spherical coordinates on $\R^{n-k}$ the Laplacian is given by $$\Delta_{\R^{n-k}} f=\frac{\partial^2 f}{\partial r^2}+\frac{n-k-1}{r} \frac{\partial f}{\partial r}+\frac 1 {r^2}\Delta_{\s^{n-k-1}} f.$$  Thus, when $f$ depends only on $r$ this reduces to $$\Delta_\Sigma f=\frac{\partial^2 f}{\partial r^2}+\frac{n-k-1}{r} \frac{\partial f}{\partial r}.$$  Likewise, from the geometric definition of the gradient we obtain $$\langle (\vec\phi,\vec x), \nabla f\rangle=rf'.$$  We will use these facts freely in the rest of the proof.

\vspace{1cm}
\noindent\emph{Proof of \ref{64}}

When $k=n-1$ and $g=g(x_n)$ we obtain $\Delta g=g''$ and $\langle (\vec\phi,x_n), \nabla g\rangle=x_n g'.$  As noted above, $|A|^2=\frac 1 2$.  We thus obtain that $$Lg=g''-\frac {x_n}2 g' + g.$$

We wish to find all such functions $g$ that satisfy $Lg=0$.  This is a second order linear ODE whose coefficients are continuous on all of $\R$, so the Existence and Uniqueness Theorem states that there must exist two linearly independent solutions defined on all of $\R$.  We saw in the proof of Proposition \ref{46} that $$g_1(x_n)=x_n^2-2$$ is one solution.  We use reduction of order to find another.
\begin{align}g_2(x_n) \notag &= (x_n^2-2)v(x_n) \\
0= L g_2(x_n) \notag &= \left(2v+4x_n v'+(x_n^2-2)v''\right) -\frac{x_n}2 \left(2x_n v+(x_n^2-2)v'\right) +(x_n^2-2)v \\
0 \notag &= 4x_n v'+(x_n^2-2)v''-\frac{x_n^3}2 v' +x_n v' \\
(x_n^2-2)v'' \notag &= \left(\frac 1 2 x_n^3-5x_n \right)v' \\
\frac{v''}{v'} \notag &= \frac {x_n} 2 \left(\frac{x_n^2-10}{x_n^2-2}\right) \\
\frac{dv'}{v'} \notag &= \left(\frac {x_n} 2 -2\left(\frac{2x_n}{x_n^2-2}\right)\right)dz \\
\log|v'| \notag &= \frac{x_n^2}4 -2\log |x_n^2-2| \\
v' \notag &= \frac{e^{\frac{x_n^2} 4}}{(x_n^2-2)^2}
\end{align}

We would like to find the second solution by setting $g_2=(x_n^2-2)v$, but there is a complication, since the formula above for $v'$ goes to $\infty$ at $x_n=\pm \sqrt 2$.  Thus, we first consider solutions on each of the three regions $(-\infty,-\sqrt 2)$, $(-\sqrt 2,\sqrt 2)$, and $(\sqrt 2,\infty)$.  Let $a$ be an element of any one of these intervals.  Then on that same interval, the general solution of the differential equation $Lg=0$ is given by
\begin{equation}\label{51}g(x_n) = A(x_n^2-2) + B(x_n^2-2)\int\limits_a^{x_n} \frac{e^{\frac{z^2}4}}{(z^2-2)^2}dz \end{equation}
where $A$ and $B$ are arbitrary constants.

We now use this to build on all of $\R$ a single solution which is linearly independent of $g_1=(x_n^2-2)$.  We begin by defining $g_2$ on the interval $(-\sqrt 2,\sqrt 2)$ by $$g_2(x_n)=(x_n^2-2)\int\limits_0^{x_n} \frac{e^{\frac{z^2}4}}{(z^2-2)^2}dz.$$
We note that this function $g_2$ is odd on its domain of definition.  By the Existence and Uniqueness Theorem, it extends uniquely to a solution defined on all of $\R$, which we will also call $g_2$.  We will see shortly that this extended function $g_2$ is odd on all of $\R$.  Currently $g_2$ is undefined at $x_n=\pm \sqrt 2$, but using L'Hospital's Rule we find the one-sided limits
\begin{align} \lim\limits_{x_n\rightarrow \sqrt 2^-}g_2 \notag &=-\frac 1 2\sqrt{\frac e 2}\\
\lim\limits_{x_n\rightarrow -\sqrt 2^+}g_2 \notag &=\frac 1 2\sqrt{\frac e 2}
\end{align}

We thus define $g_2$ to be equal to its limits at these two points, so it is now defined on the closed interval $[-\sqrt 2,\sqrt 2]$.  However, we know from Equation \ref{51} that on $(\sqrt 2,\infty)$ any solution of the differential equation is of the form $A(x_n^2-2) + B(x_n^2-2)\int\limits_2^{x_n} \frac{e^{\frac{z^2}4}}{(z^2-2)^2}dz.$  We have chosen the lower limit of integration to be $2$, but any other number in $(\sqrt 2,\infty)$ would work equally well and would only affect the value of $A$.  We now take the limit of this expression
$$\lim\limits_{x_n\rightarrow\sqrt 2^+}\left(A(x_n^2-2) + B(x_n^2-2)\int\limits_2^{x_n} \frac{e^{\frac{z^2}4}}{(z^2-2)^2}dz\right) =-B\frac 1 2\sqrt{\frac e 2}$$

This means that in order for this expression to match up with $g_2$ at $x_n=\sqrt 2$, we must have $B=1$.  The constant $A=K_1$ is fixed by requiring that the function $g_2$ be continuously differentiable at $x_n=\sqrt 2$.  We proceed similarly on the remaining interval $(-\infty,-\sqrt 2)$ to find the desired definition of $g_2$.  We obtain the constant $-K_1$ on the interval $(-\infty,-\sqrt 2)$ via symmetry considerations.

Clearly $g_2$ is an odd function, so it is independent of $g_1=x_n^2-2$.  Then any solution of the differential equation must be a linear combination of these two solutions.

\vspace{1cm}
\noindent\emph{Proof of \ref{65}}

Suppose $1\leq k\leq n-2$, and $f=f(r)$.  Then the Laplacian is given by $\Delta f=\frac{\partial^2 f}{\partial r^2}+\frac{n-k-1}{r} \frac{\partial f}{\partial r}.$  Also $|A|^2=\frac 1 2$, and $\langle (\vec\phi,\vec x), \nabla f\rangle=rf'.$  These facts combine to show that
$$Lf=f''+\left(\frac{n-k-1}r-\frac r 2\right)f'+f.$$

We wish to solve the differential equation $Lf=0$.  We know from Proposition \ref{46} that $f_1(r)=r^2-2(n-k)$ is one solution.  As we did in the proof of (1) above, we would like to use reduction of order to find a second, linearly independent solution to this differential equation.  However, in this case one of the coefficients of the differential equation is discontinuous at $r=0$.  Thus, we are only guaranteed existence of a second linearly independent solution on the intervals $(-\infty,0)$ and $(0,\infty)$.  Since we are thinking of $r$ as the distance to the origin, we restrict attention to finding this solution on the geometrically meaningful region $(0,\infty)$.

In this case reduction of order yields $$v'= \frac{e^\frac {r^2}4}{r^{n-k-1}[r^2-2(n-k)]^2}.$$  This function diverges as $r\rightarrow 0^+$, as expected.  However, the formula also diverges as $r\rightarrow \sqrt{2(n-k)}$.  Thus, it will only immediately yield a solution on the intervals $(0,\sqrt{2(n-k)})$ and $(\sqrt{2(n-k)},\infty)$.  However, using the same technique as in the proof of (1), we can piece these solutions together to obtain a solution on the full interval $(0,\infty)$.

We first note that $1\in(0,\sqrt{2(n-k)})$, so on that interval a second solution is given by
$$f_2=(r^2-2(n-k))\int\limits_1^r \frac{e^\frac {s^2}4}{s^{n-k-1}[s^2-2(n-k)]^2} ds.$$  Then we extend this solution to the endpoint $r=\sqrt{2(n-k)}$ by taking the limit.
$$\lim\limits_{r\rightarrow \sqrt{2(n-k)}^-}f_2=-\frac 1 2 \left(\frac e {2(n-k)}\right)^{\frac{n-k}2}$$

Now, on the interval $(\sqrt{2(n-k)},\infty)$, every solution of the differential equation can be written in the form
$$g(r)=A(r^2-2(n-k)) + B(r^2-2(n-k))\int\limits_{2\sqrt{2(n-k)}}^r \frac{e^\frac {s^2}4}{s^{n-k-1}[s^2-2(n-k)]^2} ds$$ for constants $A$ and $B$.  Thus, we can extend $f_2$ to be defined on all of $(0,\infty)$ by finding the correct values of $A$ and $B$ to make $f_2$ continuous and continuously differentiable.
$$\lim\limits_{r\rightarrow \sqrt{2(n-k)}^+} g(r)=-B\frac 1 2 \left(\frac e {2(n-k)}\right)^{\frac{n-k}2}$$

Thus, in order for $g(r)$ to be the continuous continuation of $f_2$, we must have $B=1$.  As in the proof of Proposition \ref{17}, we know from the Existence and Uniqueness Theorem that there exists some unique value $A=K_2$ such that the following piecewise defined function $f_2$ is a solution of the differential equation on all of $(0,\infty)$.

$$f_2(r) = \left\{
  \begin{array}{ll}
    (r^2-2(n-k))\int\limits_1^r \frac{e^\frac {s^2}4}{s^{n-k-1}[s^2-2(n-k)]^2} ds & : r \in (0,\sqrt{2(n-k)})\\
    -\frac 1 2 \left(\frac e {2(n-k)}\right)^{\frac{n-k}2} & : r = \sqrt{2(n-k)} \\
     (r^2-2(n-k))\left(K_2 + \int\limits_{2\sqrt{2(n-k)}}^r \frac{e^\frac {s^2}4}{s^{n-k-1}[s^2-2(n-k)]^2} ds\right) & : r \in (\sqrt{2(n-k)},\infty)\\
   \end{array}\right.$$

We note that the constant $K_2$ depends on the dimensional value $n-k$.  Since $f_1$ and $f_2$ are independent, any solution of the differential equation can be written as a linear combination of them.

\vspace{1cm}
\noindent\emph{Proof of \ref{66}}

Now consider the case when $k=0$.  In this case $\Sigma=\R^n$ is a flat plane through the origin, so in particular $|A|^2=0$.  We also know that $\Delta f=\frac{\partial^2 f}{\partial r^2}+\frac{n-1}{r} \frac{\partial f}{\partial r},$ and $\langle (\vec\phi,\vec x), \nabla f\rangle=rf'.$.  Putting these pieces together yields $$Lf=f''+\left(\frac{n-1}r-\frac r 2\right) f'+\frac 1 2 f.$$

We wish to solve the differential equation $Lf=0$ on the geometrically relevant region $\{r\geq 0\}$.  We note that by the Existence and Uniqueness Theorem we are only guaranteed the existence of two linearly independent solutions on the subregion $\{r>0\}$, since one of the coefficients in the differential equation diverges at $r=0$.

We begin by looking for a series solution $f=\sum\limits_{m=0}^\infty c_m r^m$.  We obtain the following recurrence relations.

\begin{align}m \notag &= 0 \ \ \ \ \ \ \ \ \ \ \ \ \ \ \ c_1=0 \\
m \notag &= 1 \ \ \ \ \ \ \ \ \ \ \ \ \ \ \ c_2=\frac{-c_0}{4n} \\
m \notag &\geq 2 \ \ \ \ \ \ \ \ \ \ \ \ c_{m+1}=\frac{m-2}{2(m+1)(n+m-1)}c_{m-1}
\end{align}
From the recurrence relations, we see that all of the odd terms $c_{2m+1}=0$, meaning that this approach only yields one linearly independent solution of the differential equation.

Solving these recurrence relations and letting $c_0=-1$, we obtain the series solution $$f_1=-1+\sum\limits_{m=1}^\infty \frac{m(2m-2)!}{2^{3m-1}(m!)^2\prod\limits_{j=0}^{m-1}(n+2j)}r^{2m}.$$  This series converges for all $r$ by the ratio test.  By making the substitution $r=|\vec x|$, it is possible to show that $f_1$ is a Jacobi function on all of $\Sigma$, including the origin.

Note that all of the non-constant terms $c_{2m}$ are positive.  Thus, this solution has a unique root, which we will call $r_1$.

As we did in the proofs of (1) and (2) above, we next find a second linearly independent solution using reduction of order.  We set $f_2(r)=v(r)f_1(r)$.  Substituting this into the equation $Lf=0$ yields

\begin{align}(2rf_1)v'' \notag &= (r^2f_1-2(n-1)f_1-4rf_1')v' \\
\frac{v''}{v'} \notag &= \frac r 2 -\frac {n-1} r -2(\log f_1)' \\
v' \notag &= \frac{e^{\frac{r^2}4}}{r^{n-1}f_1^2}
\end{align}

It is no surprise that this expression for $v'$ diverges as $r\rightarrow 0^+$, since we do not expect to find a second solution defined at $0$.  However, $v'$ also diverges as $r\rightarrow r_1$.  We work around this as we did before.  Note that we understand the solution of the differential equation on the two intervals $(0,r_1)$ and $(r_1,\infty)$.  If $a$ is any number in one of these intervals, then the general solution of the differential equation on that same interval is given by
$$f=A f_1 +B f_1\int\limits_a^r \frac{e^{\frac{s^2}4}}{s^{n-1}f_1^2}ds.$$

Clearly $\frac 1 2 r_1\in(0,r_1)$, so we define $f_2$ on that interval by $$f_2=f_1\int\limits_{\frac 1 2 r_1}^r \frac{e^{\frac{s^2}4}}{s^{n-1}f_1^2}ds.$$
We then extend the definition of $f_2$ to the point $r=r_1$ by setting it equal to its limit.  We move $f_1$ to the denominator and apply L'Hospital's Rule to show the following.  $$\lim\limits_{r\rightarrow r_1^-}f_2=-\frac{e^{\frac{r_1^2}4}}{r_1^{n-1} f_1 '(r_1)}$$  From the definition of $f_1$, we see that $f_1 '$ is a power series with all positive terms.  Thus $f_1 '(r_1)>0$, so the limit is finite.

We now extend $f_2$ onto the interval $(r_1,\infty)$ by first noting that $2r_1$ is in this interval.  Thus we know that every solution of the differential equation on this interval can be written in the form
$$\tilde f=A f_1 +B f_1\int\limits_{2r_1}^r \frac{e^{\frac{s^2}4}}{s^{n-1}f_1^2}ds.$$
We take the limit of this expression as $r\rightarrow r_1$ from the right and require that this equal the value $f_2(r_1)$.
$$\lim\limits_{r\rightarrow r_1^+}\tilde f=-B\frac{e^{\frac{r_1^2}4}}{r_1^{n-1} f_1 '(r_1)}$$

Thus $B=1$.  We know from the Existence and Uniqueness Theorem that for some choice $A=K_3$ the expression yields a continuation of $f_2$ such that the full function is a solution of the differential equation on all of $(0,\infty)$.  This solution $f_2$ is clearly not a multiple of $f_1$, so the general solution of the differential equation must be a linear combination of $f_1$ and $f_2$.  This completes the proof.
\end{proof}

We now investigate the stability of symmetric regions of $\s^k\times\R^{n-k}$.  This is accomplished by finding the regions on which the differential equation $Lf=0$ has a positive solution.  In Proposition \ref{47} we found the general radial solution of this differential equation on every unbounded cylindrical self-shrinker.  However, the behavior of these solutions is not immediately clear from their form.  Our task then is to study these solutions to see what conclusions we can draw about the regions on which positive solutions exist.  We begin by proving a proposition about unstable regions that we will need in the next chapter.

\begin{Pro}\label{17} Let $\Sigma=\s^{n-1}\times\R\subset\R^{n+1}$ be a self-shrinker.   For each $a\in [0,\sqrt 2)$, the half-infinite portion of the cylinder $\Sigma$ given by $\{x_n>a\}$ is unstable.  Further, for each such $a$ there exists some $b_a > a$ such that the portion of $\Sigma$ given by $\{a<x_n<b\}$ is unstable whenever $b>b_a$.
\end{Pro}
\begin{proof} By the first part of Lemma \ref{13}, in order to show that the portion of $\Sigma$ given by $\{x_n>a\}$ is unstable, we need only show that there is no axially symmetric Jacobi function on $\{x_n>a\}$.  However, we know from Proposition \ref{47} that any such Jacobi function $g=g(x_n)$ must be of the form $g(x_n)=c_1(x_n^2-2)+c_2 g_2,$ where $g_2$ is as defined in Proposition \ref{47}.

Since $g_2$ is odd, it has $0$ as a root.  Thus $\{|x_n|<\sqrt 2\}$ is the largest stable portion of the cylinder $\Sigma$ centered at the origin.  By Proposition \ref{9}, the portion of $\Sigma$ given by $\{x_n>0\}$ is unstable.  Thus $g_2$ must also have at least one positive root.  Let $r_0$ denote the smallest positive root of $g_2$.  Then since $g_2$ is odd, it has roots $\pm r_0$.  We note that $g_2(x_n)<0$ for all $x_n\in(0,\sqrt 2)$.  We also have that $g_2(\sqrt 2)<0$.  This implies that $r_0>\sqrt 2.$

Thus $g_2$ by itself is not strictly positive on any region $\{x_n>a\}$ where $a<\sqrt 2$.  It now suffices to show that no linear combination of $g_2$ and $g_1=x_n^2-2$ is strictly positive on any such region.  To this end, suppose $$\tilde g=Ag_1+Bg_2$$ is positive on $\{x_n>a\}$ where $a<\sqrt 2$.  Then since $g_2(r_0)=0$ and $r_0>\sqrt 2$, we must have $A>0$.  By the same logic since $g_1(\sqrt 2)=0$, we must have $B<0$.  Then on the region $\{x_n>\sqrt 2\}$ we have that
\begin{align}\tilde g \notag &= (A+BK_1)(x_n^2-2)+B(x_n^2-2)\int\limits_2^{x_n} \frac{e^{\frac{z^2}4}}{(z^2-2)^2}dz \\
\tilde g \notag &= (x_n^2-2)\left(A+BK_1+B\int\limits_2^{x_n} \frac{e^{\frac{z^2}4}}{(z^2-2)^2}dz \right)
\end{align}
However, we already know that $B<0$.  We compute the following limits.
\begin{align}
\lim\limits_{x_n\rightarrow \infty} \int\limits_2^{x_n} \frac{e^{\frac{z^2}4}}{(z^2-2)^2}dz \notag &= \infty \\
\lim\limits_{x_n\rightarrow \infty} \left(A+BK_1+B\int\limits_2^{x_n} \frac{e^{\frac{z^2}4}}{(z^2-2)^2}dz \right) \notag &= -\infty \\
\lim\limits_{x_n\rightarrow \infty} \tilde g \notag &= -\infty
\end{align}
In particular, for large enough $x_n$ we must have $\tilde g<0$, so $\tilde g$ is not strictly positive on $\{x_n>a\}$.

To see the second claim, note that since $\{x_n>a\}$ is unstable it has positive stability index.  However, by Definition \ref{16}, this means that some compact subset $\{a<x_n<b_a\}$ also has positive stability index.  This $b_a$ satisfies the second claim, because the stability index is defined as a supremum.
\end{proof}

We note that in the proof of Proposition \ref{17} above we have also proven the following corollary.

\begin{cor}\label{52} Let $\Sigma=\s^{n-1}\times\R\subset\R^{n+1}$ be a self-shrinker.  Then $C=\sqrt 2$ is the unique value such that the two $(n-1)$-surfaces $\{x_n=\pm C\}$ split $\Sigma$ into three stable regions.
\end{cor}

\begin{rem} Based on a computer approximation, $r_0\approx 3.00395$ is unique.  Whatever the exact value of $r_0$, we note that the portions of $\Sigma=\s^{n-1}\times\R$ given by $\{0<x_n<r_0\}$ and $\{-r_0<x_n<0\}$ are also stable.  By taking linear combinations of $g_1$ and $g_2$, it is possible to interpolate between the two intervals $(-\sqrt 2,\sqrt 2)$ and $(0,r_0)$ to find other intervals of comparable length over which $\Sigma$ is stable.  Also, the reflection through $0$ of each of these intervals can be obtained by interpolating between the intervals $(-\sqrt 2,\sqrt 2)$ and $(-r_0,0)$.
\end{rem}

\begin{Pro}\label{53} Let $\Sigma=\s^k\times\R^{n-k}\subset\R^{n+1}$ be a self-shrinker with $1\leq k\leq n-2$.  Then $\{r=\sqrt{2(n-k)}\}$ is the unique rotationally symmetric $(n-1)$-surface that splits $\Sigma$ into two stable regions $\{r<\sqrt{2(n-k)}\}$ and $\{r>\sqrt{2(n-k)}\}.$
\end{Pro}
\begin{proof}  Recall from Proposition \ref{47} that the general solution to the differential equation $Lf=0$ on $\{r\geq 0\}$ is given by $$f(r)=c_1(r^2-2(n-k))+c_2 f_2(r)$$ where $f_2(r)$ is only defined on $(0,\infty)$.  Thus, the only solutions defined on sets containing the origin are constant multiples of $f_1=r^2-2(n-k)$.  Thus, the largest stable region of the form $\{r<C\}$ is given by $C=\sqrt{2(n-k)}$.

We also know that the region $\{r>\sqrt {2(n-k)}\}$ is stable.  It therefore suffices to show that no linear combination of $f_1$ and $f_2$ is strictly positive on $\{r>C\}$ where $C<\sqrt{2(n-k)}$.  To see this, suppose by way of contradiction that there exist constants $A$ and $B$ and some $C<\sqrt{2(n-k)}$ such that on the region $\{r>C\}$ $$\tilde f=A(r^2-2(n-k))+Bf_2>0.$$
Recall that $f_2$ is defined to be
$$f_2(r) = \left\{
  \begin{array}{ll}
    (r^2-2(n-k))\int\limits_1^r \frac{e^\frac {s^2}4}{s^{n-k-1}[s^2-2(n-k)]^2} ds & : r \in (0,\sqrt{2(n-k)})\\
    -\frac 1 2 \left(\frac e {2(n-k)}\right)^{\frac{n-k}2} & : r = \sqrt{2(n-k)} \\
     (r^2-2(n-k))\left(K_2 + \int\limits_{2\sqrt{2(n-k)}}^r \frac{e^\frac {s^2}4}{s^{n-k-1}[s^2-2(n-k)]^2} ds\right) & : r \in (\sqrt{2(n-k)},\infty)\\
   \end{array}\right.$$

In particular $f_2\left(\sqrt{2(n-k)}\right)<0$, so we must have $B<0$.  Likewise, it is clear that $f_2(r)$ has a unique root on the region $\left\{r>\sqrt{2(n-k)}\right\}$.  Thus, we must have $A>0$.  Thus we obtain that on the region $\{r>\sqrt{2(n-k)}\}$,
\begin{align}\tilde f\notag &= A(r^2-2(n-k))+B(r^2-2(n-k))\left(K_2 + \int\limits_{2\sqrt{2(n-k)}}^r \frac{e^\frac {s^2}4}{s^{n-k-1}[s^2-2(n-k)]^2} ds\right) \\
\tilde f \notag &= (r^2-2(n-k))\left(A+B K_2 + B\int\limits_{2\sqrt{2(n-k)}}^r \frac{e^\frac {s^2}4}{s^{n-k-1}[s^2-2(n-k)]^2} ds\right)
\end{align}
We compute the following limits.
\begin{align}
\lim\limits_{r\rightarrow \infty} \int\limits_{2\sqrt{2(n-k)}}^r \frac{e^\frac {s^2}4}{s^{n-k-1}[s^2-2(n-k)]^2} ds \notag &= \infty \\
\lim\limits_{r\rightarrow \infty} \left(A+B K_2 + B\int\limits_{2\sqrt{2(n-k)}}^r \frac{e^\frac {s^2}4}{s^{n-k-1}[s^2-2(n-k)]^2} ds\right) \notag &= -\infty \\
\lim\limits_{r\rightarrow \infty} \tilde f \notag &= -\infty
\end{align}
In particular, for large enough $r$ we must have $\tilde f<0$, so $\tilde f$ is not strictly positive on $\{r>C\}$.  This completes the proof.
\end{proof}

The uniqueness results in Corollary \ref{52} and Proposition \ref{53} are not surprising, and in fact they follow from the domain monotonicity of the lowest eigenvalue.  We now give a statement of this domain monotonicity property, and we show how it implies the uniqueness results that we already proved using more direct arguments.

\begin{thm}\cite{CM1} Suppose $\Omega_1$ and $\Omega_2$ are domains in some self-shrinker $\Sigma$.  Suppose also that $\Omega_1\subset\Omega_2$, and $\Omega_2$ is strictly larger than $\Omega_1$.  Then letting $\lambda_1(\Omega_i)$ denote the lowest eigenvalue of $L$ with Dirichlet boundary condition, we have $$\lambda_1(\Omega_2)<\lambda_1(\Omega_1).$$
\end{thm}

Now note that on each of the three regions in Corollary \ref{52}, the Jacobi function $g_1=x_n^2-2$ is actually an eigenfunction.  This is because for each of the three regions, $g_1\in L_2$ with the weighted metric $d\tilde\mu=e^{-\frac{|x|^2}4}d\mu$.  Likewise, the Jacobi function on each of the two regions in Proposition \ref{53} is $f_1=r^2-2(n-k)$.  This is in $L_2$ on each of the two regions with respect to the weighted measure, so $f_1$ is an eigenfunction on each region.

The fact that these Jacobi functions are also eigenfunctions is important, because it means that $0$ is an eigenvalue of $L$ on each of the regions under consideration.  We have already shown that each of these regions is stable, meaning that $L$ has no negative eigenvalues on each region.  Thus, $0$ must be the lowest eigenvalue of $L$ on each of these regions.  Then by the domain monotonicity of the lowest eigenvalue, making any of the regions larger would decrease the lowest eigenvalue below $0$, so the region would no longer be stable.

The situation is quite different when $\Sigma=\R^n$ is a hyperplane.  Recall from Proposition \ref{47} that the only globally defined Jacobi functions are constant multiples of $f_1$, where
$$f_1(r)=-1+\sum\limits_{m=1}^\infty \frac{m(2m-2)!}{2^{3m-1}(m!)^2\prod\limits_{j=0}^{m-1}(n+2j)}r^{2m}.$$  This function clearly has exactly one root when $r\geq 0$.  As before, we continue to call this root $r_1$.  Then the $(n-1)$-sphere $\{r=r_1\}$ splits $\Sigma$ into two stable regions given by $\{r<r_1\}$ and $\{r>r_1\}$.  Since $f_1$ is bounded on compact sets, we clearly have $f_1\in L_2(\{r<r_1\})$ with the weighted metric.  Thus, $0$ is the smallest eigenvalue on this region.  However, we now show that $f_1$ is not in the weighted $L_2$ space on $\{r>r_1\}$, and hence $f_1$ is not an eigenfunction on the unbounded region.

\begin{lem}\label{62} Let $\Sigma=\R^n\subset\R^{n+1}$ be a self-shrinker.  Define
$$f_1(r)=-1+\sum\limits_{m=1}^\infty \frac{m(2m-2)!}{2^{3m-1}(m!)^2\prod\limits_{j=0}^{m-1}(n+2j)}r^{2m},$$ and let $r_1$ denote the positive root of $f_1$.  Then $f_1$ is a positive Jacobi function on $\{r>r_1\}$, but it is not an eigenfunction on that same region.
\end{lem}
\begin{proof} We know from Proposition \ref{47} that $Lf_1=0$, so it is a Jacobi function.  It is clearly positive on $\{r>r_1\}$.  It thus suffices to show that $f_1$ is not an eigenfunction on that same region.  To do this, we will show that $f_1$ is not in the weighted $L_2$ space on this region.

Letting $C_n$ denote the volume of the unit $n$-sphere, we obtain the following expressions for the $L_2$ norm of $f_1$.
\begin{align}\|f_1\|_{L^2} \notag &= C_n\int\limits_0^\infty r^n f_1^2 e^{-\frac {r^2}4} dr \\
\|f_1\|_{L^2} \notag &= C_n\int\limits_0^\infty r^n \left(f_1 e^{-\frac {r^2}8}\right)^2 dr
\end{align}
We wish to show that this integral diverges.  For this, it is sufficient to show that the integrand does not go to $0$ as $r\rightarrow\infty$.  We will show more, namely that the integrand actually diverges to $\infty$ as $r\rightarrow\infty$.  To show this, we will show that for large enough $r$ $$\left(f_1 e^{-\frac {r^2}8}\right)>1.$$
Clearly this is equivalent to the claim that for large enough $r$, we have $f_1>e^{\frac{r^2}8}$.  For a fixed dimension $n$, we define the constants $a_{2m}$ and $b_{2m}$ by the following equations.
$$f_1(r)=-1+\sum\limits_{m=1}^\infty \frac{m(2m-2)!}{2^{3m-1}(m!)^2\prod\limits_{j=0}^{m-1}(n+2j)}r^{2m}=\sum\limits_{m=0}^\infty a_{2m}r^{2m}$$
$$e^{\frac{r^2}8}=\sum\limits_{m=0}^\infty \frac 1{2^{3m} m!} r^{2m} = \sum\limits_{m=0}^\infty b_{2m}r^{2m}$$

We will show that \begin{equation}\label{54}\lim\limits_{m\rightarrow\infty} \frac{a_{2m}}{b_{2m}}=\infty.\end{equation}  We now assume this fact and show how this completes the proof.

Assuming Equation \ref{54}, we see there exists some $M$ such that $a_{2m}>2 b_{2m}$ for all $m\geq M$.  Then for $r\geq 1$,
\begin{align} f_1-e^{\frac{r^2}8} \notag &= \sum\limits_{m=0}^\infty a_{2m} r^{2m} - \sum\limits_{m=0}^\infty b_{2m} r^{2m} \\
f_1-e^{\frac{r^2}8} \notag &= \sum\limits_{m=0}^\infty (a_{2m} - b_{2m}) r^{2m} \\
f_1-e^{\frac{r^2}8} \notag &= \sum\limits_{m=0}^{M-1} (a_{2m} - b_{2m}) r^{2m} + \sum\limits_{m=M}^\infty (a_{2m} - b_{2m}) r^{2m} \\
f_1-e^{\frac{r^2}8} \notag &\geq \sum\limits_{m=0}^{M-1} -|a_{2m} - b_{2m}| r^{2m} + \sum\limits_{m=M}^\infty b_{2m} r^{2m} \\
f_1-e^{\frac{r^2}8} \notag &\geq \left(\sum\limits_{m=0}^{M-1} -|a_{2m} - b_{2m}|\right) r^{2M-2} + \left(\sum\limits_{m=M}^\infty b_{2m}\right) r^{2M} \\
f_1-e^{\frac{r^2}8} \notag &\geq (-A + Br^2)r^{2M-2}
\end{align}

In the above, $A$ and $B$ are positive constants.  Thus, for large enough $r$, $f_1>e^{\frac{r^2}8}$.  Thus, $$f_1 e^{-\frac{r^2}8}\geq  e^{\frac{r^2}8} e^{-\frac{r^2}8}=1.$$  This completes the proof.

Thus, it suffices to prove Equation \ref{54}.  We will do this in two cases, depending on whether the dimension $n$ is even or odd.  Suppose first that $n$ is even.  Then there exists an integer $d$ such that $n=2d$.  We restrict attention to $m\geq 1$ and compute.

\begin{align} \prod\limits_{j=0}^{m-1} (n+2j) \notag &= (2d)(2d+2)\cdots(2d+2(m-1)) \\
\prod\limits_{j=0}^{m-1} (n+2j) \notag &= 2^m d(d+1)(d+2)\cdots (d+m-1) \\
\prod\limits_{j=0}^{m-1} (n+2j) \notag &= \frac{2^m(d+m-1)!}{(d-1)!} \\
\frac {a_{2m}}{b_{2m}} \notag &= \frac{m(2m-2)!}{2^{3m-1}(m!)^2\prod\limits_{j=0}^{m-1}(n+2j)} 2^{3m} m! \\
\frac {a_{2m}}{b_{2m}} \notag &= \frac{2m(2m-2)!}{(m!)\prod\limits_{j=0}^{m-1}(n+2j)} \\
\frac {a_{2m}}{b_{2m}} \notag &= \frac{2m(2m-2)!(d-1)!}{(m!)2^m(d+m-1)!} \\
\frac {a_{2m}}{b_{2m}} \notag &= \frac{(2m)!(d-1)!(d+m)}{(m!)(d+m)!2^m(2m-1)}
\end{align}

We now apply Stirling's Approximation.  For large $m$ $$m!\approx \sqrt{2\pi m}\left(\frac m e\right)^m.$$
This approximation is valid in the limit, which is all we care about.

\begin{align} \lim\limits_{m\rightarrow\infty}\frac {a_{2m}}{b_{2m}} \notag &= \lim\limits_{m\rightarrow\infty} \frac{2\sqrt{\pi m}\left(\frac{2m}e\right)^{2m}(d-1)!(d+m)}{\sqrt{2\pi m}\left(\frac{m}e\right)^m \sqrt{2\pi(d+m)}\left(\frac{d+m}e\right)^{d+m}2^m (2m-1)} \\
\lim\limits_{m\rightarrow\infty}\frac {a_{2m}}{b_{2m}} \notag &= \lim\limits_{m\rightarrow\infty} \frac{\left(\frac{2m}e\right)^{m}(d-1)!(d+m)}{ \sqrt{\pi(d+m)}\left(\frac{d+m}e\right)^{d+m} (2m-1)} \\
\lim\limits_{m\rightarrow\infty}\frac {a_{2m}}{b_{2m}} \notag &= \frac{(d-1)!e^d}{\sqrt \pi}\lim\limits_{m\rightarrow\infty} \frac {2^m}{(d+m)^{d-\frac 1 2}} \left(\frac m{d+m}\right)^m \\
\lim\limits_{m\rightarrow\infty}\frac {a_{2m}}{b_{2m}} \notag &= \frac{(d-1)!e^d}{\sqrt \pi}\lim\limits_{m\rightarrow\infty} \frac {2^m}{(d+m)^{d-\frac 1 2}} \frac 1{(1+\frac d m)^m} \\
\lim\limits_{m\rightarrow\infty}\frac {a_{2m}}{b_{2m}} \notag &= \frac{(d-1)!}{\sqrt \pi}\lim\limits_{m\rightarrow\infty} \frac {2^m}{(d+m)^{d-\frac 1 2}} \\
\lim\limits_{m\rightarrow\infty}\frac {a_{2m}}{b_{2m}} \notag &= \infty
\end{align}

This proves Equation \ref{54} whenever the dimension $n$ is even.  When $n$ is odd, we notice that
$$\frac{2m(2m-2)!}{(m!)\prod\limits_{j=0}^{m-1}(n+2j)} > \frac{2m(2m-2)!}{(m!)\prod\limits_{j=0}^{m-1}(n+1+2j)}.$$  However, we already know from the above computation that this expression goes to $\infty$ as $m\rightarrow \infty$.  This proves Equation \ref{54} in the case when $n$ is odd, which completes the proof.
\end{proof}

Since $f_1$ is not an eigenfunction on the region $\{r>r_1\}$ in the hyperplane, it is possible that the smallest eigenvalue of $L$ on this region is strictly positive.  In this case, it would be possible to find some value $C<r_1$ such that the region in the hyperplane given by $\{r>C\}$ is stable. We already know that any subset of the ball $\{r<r_1\}$ is stable, so this would show that any $(n-1)$-sphere between $\{r=C\}$ and $\{r=r_1\}$ splits the hyperplane into two stable region.

In the following chapter we will use a very different argument to show that this phenomenon does occur in the case of a flat plane in $\R^3$.  In particular, we will show that the region in the $2$-plane given by $\{r>\sqrt 2\}$ is stable (see Theorem \ref{49}).  In Remark \ref{55} we estimate the value of $r_1$ for the $2$-plane to be about $2.514$.  Thus, any circle centered at the origin with radius between $\sqrt 2$ and $r_1\approx 2.514$ splits the plane into two stable regions.

We note that $f_1$ is not an eigenfunction on $\{r>r_1\}$ in a hyperplane of any dimension, and also the index of every hyperplane is $1$.  These facts lead us to conjecture that every hyperplane has a similar non-uniqueness property.  However, our argument in the case of a plane will depend on a special convergence result in $\R^3$, so we are unable at the present time to generalize these results to higher dimensions.

We note that in spite of our interest in the value $r_1$, we have yet to investigate it directly.  We do this in the following proposition and remark.

\begin{Pro}\label{48} Let $\Sigma=\R^n\subset\R^{n+1}$ be a self-shrinker.  The portion of $\Sigma$ given by $$\left\{r>2\sqrt{\left( n+2\right)\left(\sqrt{\frac{3n+2}{n+2}}-1\right)}\right\}$$ is stable.
\end{Pro}
\begin{proof} By Proposition \ref{47} we know that $L f_1=0$, where $$f_1(r)=-1+\sum\limits_{m=1}^\infty \frac{m(2m-2)!}{2^{3m-1}(m!)^2\prod\limits_{j=0}^{m-1}(n+2j)}r^{2m}.$$
As we mentioned in that proof, all of the terms in this power series are positive except the constant term.  Thus there exists a single positive root of the series, and we call that root $r_1$.  We also see that all of the partial sums are strictly less than $f_1$ for $r>0$.  Thus, if we find a root of a partial sum, then we know the root $r_1$ of the full series must be lower.  We are thus able to estimate $r_1$ from above by considering only the first few terms of $f_1$.

To second order, $f_1(r)=-1+\frac 1 {4n} r^2$.  This has the positive root $r=2\sqrt{n}$.

To fourth order, $f_1(r)=-1+\frac 1 {4n} r^2+\frac 1 {2^5 n(n+2)}r^4$.  This has the positive root $$r=2\sqrt{(n+2)\left(\sqrt{\frac{3n+2}{n+2}}-1\right)}.$$

Thus, $f_1$ is a positive Jacobi function of $L$ on the region of $\Sigma$ given by $$\left\{r>2\sqrt{\left( n+2\right)\left(\sqrt{\frac{3n+2}{n+2}}-1\right)}\right\}$$ which completes the proof.
\end{proof}

\begin{rem}\label{55}  In the above proof, we approximate the root $r_1$ of the power series $f_1$ using partial sums.  We now use a computer program to find a numerical approximation of $r_1$ for the first few values of $n$.  We compare these to the partial sum approximations from Proposition \ref{48} in the following table.

In the first column we show the dimension $n$ of the hyperplane.  In the second column we give a numerical value of the second order approximation $r_1\approx 2\sqrt{n}.$  In the third column we give a numerical value of the fourth order approximation $r_1\approx 2\sqrt{(n+2)\left(\sqrt{\frac{3n+2}{n+2}}-1\right)}$.  In the last column we give a numerical approximation of $r_1$ obtained using Mathematica.  All numbers are rounded to three decimal places.
\vspace{1cm}
\begin{center}
\begin{tabular}{|c|c|c|c|}
\hline
\multicolumn{4}{|c|}{Approximate values of $r_1$} \\
\hline
n & 2nd order & 4th order & Full \\ \hline \hline
2 & 2.828 & 2.574 & 2.514  \\ \hline
3 & 3.464 & 3.109 & 3.004 \\ \hline
4 & 4     & 3.558 & 3.408 \\ \hline
5 & 4.472 & 3.954 & 3.760 \\ \hline
6 & 4.899 & 4.312 & 4.076 \\ \hline
7 & 5.292 & 4.642 & 4.364 \\ \hline
\end{tabular}
\end{center}
\vspace{1cm}

\end{rem}

We now turn our attention to the situation in $\R^3$.  In this setting more is known, so we can obtain improved results.

\newpage

\chapter{Half-Infinite, Stable Self-Shrinkers with Boundary in $\R^3$}\label{60}

\section{Background Results}

In this chapter we look at circular slices of the self-shrinking cylinder $\s^1\times\R\subset\R^3$.  We let the $x_3$-axis be the axis of symmetry of the cylinder.

\begin{defn} Let $\Sigma=\s^{1}\times\R$ be a self-shrinker.  For each $a\in\R$, we let $$\gamma_a=\Sigma\cap\{x_3=a\}.$$
\end{defn}

For each of these slices $\gamma_a$ we will find a stable, half-infinite self-shrinker with boundary $\gamma_a.$  Recall from Corollary \ref{52} that when $a\geq \sqrt 2$ the portion of the cylinder given by $\{x_3>a\}$ is stable.  Likewise when $a\leq -\sqrt 2$ we have the stable region $\{x_3< a\}$.  We will spend the rest of this chapter finding a stable half-infinite self-shrinker with boundary $\gamma_a$ for the remaining values of $a$.  In Section \ref{56} we will show the existence of such a surface for each $a\in (0,\sqrt 2)$.  By symmetry this also solves the problem for $a\in(-\sqrt 2, 0)$.  In Section \ref{57} we will analyze these surfaces and use them to show that the portion of the plane given by $\{|\vec x|>\sqrt 2\}$ is stable.

We begin by stating results that hold in as much generality as possible, but we will be forced to specialize to the low dimensional case before long. We first wish to define a class of objects which play a fundamental role in geometric measure theory, the class of rectifiable currents.  This in turn requires the introduction of Hausdorff measure.

\begin{defn} Let $A\subset \R^n$, and define the diamater of $A$ to be
$$diam(A)=\sup{|x-y|:x,y\in A}$$
Let $\alpha_m$ equal the volume of the unit ball in $\R^m$.  Then the $m-$dimensional Hausdorff measure of $A$ is defined by $$\textbf{H}^m(A)=\lim\limits_{\delta\rightarrow 0^+}\inf\limits_{\substack{
A\subset\cup S_j \\
diam(S_j)\leq\delta}}\sum\limits_{j}\alpha_m\left(\frac{diam(S_j)} 2\right)^m$$
The Hausdorff dimension of $A$ is given by $$dim_\textbf{H}(A)=\inf\{m\geq 0:\textbf{H}^m(A)=0\}$$
\end{defn}

\begin{defn}\label{25} Consider a Borel set $B\subset \R^n$.  We say $B$ is $(\textbf{H}^m,m)$ rectifiable if it has the following properties.
\begin{enumerate}
\item There exist at most countably many bounded subsets $K_i\subset\R^m$ and Lipschitz maps $f_i:\R^m\rightarrow\R^n$ such that $B=\cup f_i(K_i)$ (ignoring sets of measure 0).
\item $\textbf{H}^m(B)<\infty$
\end{enumerate}

A rectifiable $m$-current is a compactly supported, oriented $(\textbf{H}^m,m)$ rectifiable set with integer multiplicities.
\end{defn}

\begin{rem} Currents are usually defined as linear functionals on differential forms.  We are not explicitly stating the definition in that form.  However, it is possible to integrate a smooth differential form $\phi$ over a rectifiable current defined above.  In this way, each rectifiable current $B$ gives rise to the following linear functional on differential forms.
$$\phi\rightarrow\int\limits_B \phi$$
The inclusion of integer multiplicities in definition \ref{25} can then be seen as necessary to allow for the standard additivity of linear functionals.
\end{rem}

It was necessary to define rectifiable currents, because these are the basic objects dealt with in geometric measure theory.  The general approach often taken to proving an existence result such as the one we are pursuing is to first show the existence of a rectifiable current with the desired minimization property.

However, as can be seen from definition \ref{25}, rectifiable currents are extremely general objects which often bear little resemblance to our usual notion of a surface.  Thus it is then necessary to show that the obtained minimizing current is regular, or smooth.  Otherwise, the object obtained will still be an area minimizer, but it cannot be considered a minimal surface.

Our argument will be a little more complicated.  We will first show the existence and regularity of a sequence of self-shrinkers in increasingly large compact sets.  We will then show that there exists a subsequence of these self-shrinkers which converges to the desired half-infinite self-shrinker.  The following general theorems of Federer will form the foundation of the first steps in this argument.

\begin{thm}\label{27}\cite{F} Suppose $T$ is a rectifiable current in a compact, $C^1$ Riemannian manifold $M$.  Then consider the set of all rectifiable currents $S$ in $M$ such that $\partial S=\partial T$.  There exists at least one such current $S$ of least area.
\end{thm}

Note that the above theorem holds in all dimensions.  However, the following regularity result is more restrictive.

\begin{thm}\label{28}\cite{F} Let $T\subset \R^n$ be an $(n-1)$ dimensional, area minimizing rectifiable current.  Then in its interior $T$ is a smooth, embedded manifold except for a singular set of Hausdorff dimension at most $(n-8)$.  In particular, if $n\leq 7$ then $T$ has no interior singularities.
\end{thm}

\section{Existence of the Surfaces}\label{56}

We now show that there exists some stable half-infinite MCF self-shrinker with boundary $\gamma_a$ for each $a\in [0,\sqrt 2)$.  We will need a few well known results from elsewhere which we will cite and use without proof.

Note that by Theorem \ref{11}, MCF self-shrinkers in $\R^3$ can also be considered as minimal surfaces with respect to the conformal metric $$d\tilde\mu=e^{\frac{-|\vec x|^2} 4}d\mu.$$  This is because with respect to this metric $F_{0,1}$ is just a multiple of the area functional.  Then by Theorem \ref{12}, the stability of a self-shrinker is equivalent to its stability as a minimal surface with respect to $d\tilde\mu.$  We can therefore appeal directly to the known existence and regularity results for stable minimal surfaces in order to show existence and regularity of stable self-shrinkers.

Our approach will be to construct a sequence of compact, stable self-shrinkers in increasingly large domains.  We will obtain a sequence of surfaces with one boundary component fixed at $\gamma_a$ and the other running off to infinity along $C$.  We will then show that some subsequence of these surfaces converges to a stable, half-infinite self-shrinker, and furthermore this limit surface is axially symmetric.  The existence of a convergent subsequence will follow from Theorem \ref{21} below, which only holds in three dimensional manifolds.  As such, for the remainder of this chapter we will restrict attention to surfaces in $\R^3$.

In the following lemma, we will construct the sequence of self-shrinkers with one boundary component $\gamma_a$ and the other a parallel, coaxial circle $\gamma_b$.  We wish to exclude the possibility that for some choice of $b$ our solution splits into two disconnected topological discs, one with boundary $\gamma_a$ and the other with boundary $\gamma_b$.  To accomplish this, we will solve the Plateau problem in a domain from which the interior of $C$ has been removed.

\begin{lem}\label{18} Fix $a\in [0,\sqrt 2)$, and let $b_a$ be given by Proposition \ref{17}.  Let $b\in\mathbb{Z}$ satisfy $b>b_a$.  Let $\Omega_b=\{(x,y,z)\in\R^3:x^2+y^2+z^2\leq b^2+2, x^2+y^2\geq 2\}$.  Then there exists a stable, embedded self-shrinker $\Sigma_b\subset \Omega_b$ such that $\partial\Sigma_b=\gamma_a\cup\gamma_b$.
\end{lem}
\begin{proof}  We view $\Omega_b$ as a subset of $\R^3$ endowed with the metric $d\tilde\mu=e^{\frac{-|\vec x|^2} 4}d\mu$ from Theorem \ref{11}.  Note that the portion of the cylinder given by $\{a<x_3<b\}$ is a rectifiable current in $\Omega_b$ with boundary $\gamma_a\cup\gamma_b$.  Thus, the existence of a least area current $\Sigma_b\subset\Omega_b$ with the same boundary is guaranteed by Theorem \ref{27}.  The fact that this current is also a smooth embedded surface with no interior singular points follows from Theorem \ref{28}.  Area minimizing surfaces with respect to $d\tilde\mu$ are stable minimal surfaces, so $\Sigma_{b}$ is a stable, embedded self-shrinker with $\partial\Sigma_{b}=\gamma_a\cup\gamma_b.$
\end{proof}

We wish to show the convergence of a subsequence of $\{\Sigma_b\}$, but for this we need a uniform bound on the curvature of the $\Sigma_b$.  The following theorem of Schoen gives this bound away from the boundary.

\begin{thm}\label{19}\cite{S} Let $\Sigma$ be an immersed stable minimal surface with trivial normal bundle, and let $B_{r_0}\subset \Sigma\setminus\partial\Sigma$ be an intrinsic open ball of radius $r_0$ contained in $\Sigma$ and not touching the boundary of $\Sigma$.  Then there exists a constant $C$ such that for any $\sigma>0$, $$\sup\limits_{B_{r_0-\sigma}}|A|^2 \leq C\sigma^{-2}.$$
\end{thm}

\begin{lem}\label{20} For large enough $b$, and away from the boundary component $\gamma_b$, the family of surfaces $\Sigma_b$ from Lemma \ref{18} have uniformly bounded $|A|^2$ up to and including the boundary component $\gamma_a$.
\end{lem}
\begin{proof} If we fix $\sigma>0$, we obtain a bound on $|A|^2$ for each $\Sigma_b$ further than $\sigma$ away from $\gamma_a$.  We still need to worry about something going wrong near $\gamma_a$.  However, we know from Hardt-Simon \cite{HS} that each $\Sigma_b$ is regular up to the boundary.  Thus, we know that each $\Sigma_b$ has bounded $|A|^2$ up to the boundary.  It thus suffices to rule out the possibility that these bounds on $|A|^2$ are themselves unbounded.  That is, we need to show that a single bound holds for every $\Sigma_b$ in the sequence.

We aren't concerned about the behavior of $\Sigma_b$ near the boundary component $\gamma_b$, because these components don't show up in the limit.  As such, for the rest of this proof, we will restrict attention to the interior of a large ball in $\R^3$, say the ball of radius 20.  We will also assume that $b>30$, so that we are sufficiently far away from $\gamma_b$.  Thus, Schoen's curvature estimate from Theorem \ref{19} holds for every $\Sigma_b$ except near $\gamma_a$.

Suppose by way of contradiction that there exists a subsequence $\Sigma_i$ such that $$\max\limits_{\Sigma_i}|A|>i.$$  For each $i$, pick a point $x_i\in\Sigma_i$ such that $|A(x_i)|=\max\limits_{\Sigma_i}|A|$.  Then by Theorem \ref{19}, as $i$ increases, the $x_i$ must tend toward $\gamma_a$.

Define a new sequence $$\Gamma_i=|A(x_i)|(\Sigma_i-x_i).$$  Then this sequence has the property that each $\Gamma_i$ has $|A(0)|=1$ and $|A|\leq 1$ everywhere else.  Then since the $\Sigma_i$ are stable and embedded, the new surfaces $\Gamma_i$ are also embedded and stable with respect to the dilated metric.  These dilated metrics are becoming more and more flat as $i\rightarrow\infty$.  Thus by Arzela-Ascoli, some subsequence of the $\Gamma_i$ converges to a limit surface $\Gamma.$  This surface is embedded and stable as a surface in $\R^3$.  It also has $|A(0)|=1$.

There are now two cases, depending on what happens to $\gamma_a$ in the limit.  Case 1 is that $\gamma_a$ runs off to $\infty$ and does not show up in the limit.  In that case $\Gamma$ is a complete, stable minimal surface in $\R^3.$  However, Bernstein's Theorem then says that $\Gamma$ is a flat plane.  This contradicts the fact that $\Gamma$ has $|A(0)|=1$, since a flat plane has $|A|\equiv 0$.

Case 2 is that $\gamma_a$ does show up in the limit as the boundary of $\Gamma$.  In this case, note that the limit of $\gamma_a$ under our successive dilations is a straight line.  Then we have that $\Gamma$ is an embedded, area minimizing surface in $\R^3$ whose boundary is a straight line.  Then by a result of P\'{e}rez \cite{P}, $\Gamma$ is a half-plane.  However, this implies that $\Gamma$ has $|A|\equiv 0$, which contradicts the fact that $|A(0)|=1$.  This completes the proof.
\end{proof}

We now state a well-known compactness theorem stated by Anderson in \cite{A} that gives us convergence of a subsequence of $\{\Sigma_b\}$.

\begin{thm}\label{21}\cite{A} Let $\Omega$ be a bounded domain in a complete Riemannian 3-manifold $N^3$, and let $M_i$ be a sequence of minimally immersed surfaces in $\Omega$.  Suppose there is a constant $C$ such that the Gauss curvature $K_{M_i}(x)$ satisfies $|K_{M_i}(x)|<C$ for all $i$.  Then a subsequence of $\{M_i\}$ converges smoothly (in the $C^k$ topology, $k\geq 2$) to an immersed minimal surfaces $M_\infty$ (with multiplicity) in $\Omega$, and $|K_{M_\infty}(x)|\leq C$.  If each $M_i$ is embedded, then $M_\infty$ is also embedded.
\end{thm}

\begin{thm}\label{22} There exists a subsequence of the surfaces $\Sigma_b$ from Lemma \ref{18} that converges to some limit surface $\Sigma_\infty$.  This limit surface $\Sigma_\infty$ is a stable, half-infinite, embedded self-shrinker with boundary $\gamma_a$.
\end{thm}
\begin{proof} Recall from Lemma \ref{18} that $\Omega_b=\{(x,y,z)\in\R^3:x^2+y^2+z^2\leq b^2+2, x^2+y^2\geq 2\}$.  By Theorem \ref{21}, for any fixed $\Omega_{b_0}$, there exists a subsequence of $\{\Sigma_b\}$ that converges smoothly on $\Omega_{b_0}$.  We can then restrict attention to this subsequence, and find a further subsequence that converges on $\Omega_{b_0+1}$.  Repeating this process, we obtain a sequence of sequences.  We can then take the diagonal elements to form a single subsequence.  This subsequence converges to a limit surface $\Sigma_\infty$ with the desired properties.
\end{proof}

\section{Analysis of the Surfaces}\label{57}

We now state without proof a well known theorem that we will need in the following proof.

\begin{thm}\label{31} Sard's Theorem.

Given a smooth function $h$ from one manifold to another, the image of the set of critical points of $h$ has Lebesgue measure $0$.
\end{thm}

\begin{thm}\label{32} The surface $\Sigma_\infty$ from Theorem \ref{22} is rotationally symmetric about the $z$-axis.
\end{thm}
\begin{proof} Let $\vec{v}$ be a vector field on $\Sigma_\infty$ given by rotation in a fixed direction about the $z$-axis.  For concreteness, at a point $(x,y,z)\in\Sigma_\infty$, let $\vec v=(-y,x,0)$.  At every point of $\Sigma_\infty$, let $\vec n$ denote the outward pointing unit normal.

We define the function $f:\Sigma_\infty\rightarrow \R$ by $f=<\vec v, \vec n>$.  Note that on the boundary $\gamma_a$, $f\equiv 0$.  We first show that $Lf=0$ which implies that $f$ is either identically $0$, or it is an eigenfunction of $L$ with eigenvalue $0$.  Recall from Definition \ref{8} that $$Lf=\Delta f-\frac 1 2 \langle \vec{x},\nabla f\rangle+(|A|^2+\frac 1 2)f.$$  Let $\{\vec{e_i}\}_{i=1}^2$ be an orthonormal frame for the tangent space to $\Sigma_\infty$.  We compute $Lf$ term by term, starting with $\Delta f$.  In the following computations, we follow Einstein's convention of summing over repeated indices.

\begin{align} f \notag &= \langle\vec v, \vec n\rangle \\
 \nabla_{\vec{e_i}}f \notag &= \langle\nabla_{\vec{e_i}}\vec v,\vec n\rangle+\langle\vec v, \nabla_{\vec{e_i}}\vec n\rangle \\
 \nabla_{\vec{e_i}}f \notag &= \langle\nabla_{\vec{e_i}}\vec v, \vec n\rangle-a_{ij}\langle\vec v,\vec{e_j}\rangle \\
 \nabla_{\vec{e_k}}\nabla_{\vec{e_i}}f \notag &= \langle\nabla_{\vec{e_k}}\nabla_{\vec{e_i}}\vec v, \vec n\rangle - a_{kj}\langle\nabla_{\vec{e_i}}\vec v, \vec{e_j}\rangle-a_{ij}\langle\nabla_{\vec{e_k}}\vec v,\vec{e_j}\rangle \\ \notag & \ \ \ -a_{ij,k}\langle\vec v, \vec{e_j}\rangle-a_{ij}\langle\vec v, \nabla_{\vec{e_k}}\vec{e_j}\rangle \\
 \Delta f \notag &= \langle\Delta \vec v, \vec n\rangle - 2a_{ij}\langle\nabla_{\vec{e_i}}\vec v, \vec{e_j}\rangle-a_{ii,j}\langle\vec v,\vec{e_j}\rangle - a_{ij}\langle\vec v,a_{ij}\vec n\rangle \\
 \Delta f \notag &= \langle\Delta \vec v, \vec n\rangle - 2a_{ij}\langle\nabla_{\vec{e_i}}\vec v, \vec{e_j}\rangle + \langle\vec v,\nabla H\rangle - |A|^2 f
\end{align}

We now use the fact that $\vec v$ is a Killing field, so for all vector fields $\vec X, \vec Y$ we have that $\langle \nabla_{\vec X}\vec v,\vec Y\rangle=-\langle \nabla_{\vec Y}\vec v,\vec X\rangle$. Thus

\begin{align} \langle \nabla_{\vec {e_i}}\vec v,\vec {e_i}\rangle \notag &= 0 \\
 a_{ij}\langle \nabla_{\vec {e_i}}\vec v,\vec {e_j}\rangle \notag &= \sum\limits_{i\neq j}a_{ij}\langle \nabla_{\vec {e_i}}\vec v,\vec {e_j}\rangle \\
 a_{ij}\langle \nabla_{\vec {e_i}}\vec v,\vec {e_i}\rangle \notag &= \sum\limits_{i < j}(a_{ij}-a_{ji})\langle \nabla_{\vec {e_i}}\vec v,\vec {e_j}\rangle \\
 a_{ij}\langle \nabla_{\vec {e_i}}\vec v,\vec {e_i}\rangle \notag &= 0
\end{align}

The last equality follows from the symmetry of the second fundamental form $A$.  Recall that since $\vec v$ is a rotation vector field, it is linear on $\R^3$.  Thus $\Delta_{\R^3}\vec v=0.$  However, we are working with the tangential Laplacian $\Delta_{\Sigma}$.  We compute

\begin{align} \Delta \vec v \notag &= \Delta_{\Sigma}\vec v = \Delta_{\R^3}\vec v - \nabla_{\nabla_{e_i}e_i}\vec v - \nabla_{\vec n}\nabla\vec v \\
\Delta \vec v \notag &= 0 - \nabla_{\nabla_{e_i}e_i}\vec v - 0 \\
\Delta \vec v \notag &= - \nabla_{a_{ii}\vec n}\vec v \\
\Delta \vec v \notag &= - a_{ii}\nabla_{\vec n}\vec v \\
\langle \Delta \vec v,\vec n\rangle \notag &= -a_{ii}\langle \nabla_{\vec n}\vec v, \vec n\rangle \\
\langle \Delta \vec v,\vec n\rangle \notag &= 0
\end{align}

Putting these pieces together yields $$\Delta f=\langle\vec v,\nabla H\rangle - |A|^2 f.$$

We now turn our attention to the second term in $Lf.$  Recall from Lemma \ref{30} that since $\Sigma_\infty$ is a self-shrinker its mean curvature $H$ satisfies $H=\frac 1 2 \langle \vec x, \vec n\rangle$.  We will use this fact in the following.

\begin{align}H \notag &= \frac 1 2 \langle \vec x, \vec n\rangle \\
\nabla_{\vec{e_i}}H \notag &= \frac 1 2 \langle \nabla_{\vec{e_i}}\vec x, \vec n\rangle + \frac 1 2 \langle \vec x, \nabla_{\vec{e_i}}\vec n\rangle \\
\nabla_{\vec{e_i}}H \notag &= \frac 1 2 \langle \vec{e_i}, \vec n\rangle - \frac 1 2 a_{ij}\langle \vec x, \vec{e_j}\rangle \\
\nabla_{\vec{e_i}}H \notag &= -\frac {1} 2 a_{ij}\langle \vec x, \vec{e_j}\rangle \\
\nabla H \notag &= -\frac {1} 2 a_{ij}\langle \vec x, \vec{e_j}\rangle \vec{e_i} \\
\langle \vec v,\nabla H\rangle \notag &= -\frac {1} 2 a_{ij}\langle \vec x, \vec{e_j}\rangle \langle\vec v,\vec{e_i}\rangle \\
 \nabla f \notag &= \nabla \langle \vec v,\vec n\rangle \\
\nabla_{\vec{e_i}}f \notag &= \langle \nabla_{\vec{e_i}}\vec v, \vec n \rangle +\langle \vec v, \nabla_{\vec{e_i}}\vec n\rangle \\
\nabla_{\vec{e_i}}f \notag &= \langle \nabla_{\vec{e_i}}\vec v, \vec n \rangle -a_{ij}\langle \vec v, \vec{e_j}\rangle \\
\nabla f \notag &= \langle \nabla_{\vec{e_i}}\vec v, \vec n \rangle\vec{e_i} -a_{ij}\langle \vec v, \vec{e_j}\rangle\vec{e_i} \\
\frac 1 2 \langle\vec x,\nabla f\rangle \notag &= \frac 1 2 \langle \nabla_{\vec{e_i}}\vec v, \vec n \rangle \langle\vec x,\vec{e_i}\rangle -\frac 1 2 a_{ij}\langle \vec v, \vec{e_j}\rangle \langle \vec x, \vec{e_i}\rangle \\
\frac 1 2 \langle\vec x,\nabla f\rangle \notag &= \frac 1 2 \langle \nabla_{\vec{e_i}}\vec v, \vec n \rangle \langle\vec x,\vec{e_i}\rangle +\langle \vec v, \nabla H\rangle
\end{align}

We now have that $$Lf= - \frac 1 2 \langle \nabla_{\vec{e_i}}\vec v, \vec n \rangle \langle\vec x,\vec{e_i}\rangle +\frac 1 2 \langle \vec v,\vec n\rangle.$$  We investigate the first term, again using that $\vec v$ is a Killing field.

\begin{align} \langle \nabla_{\vec{e_i}}\vec v, \vec n \rangle \langle\vec x,\vec{e_i}\rangle \notag &= - \langle \nabla_{\vec n}\vec v, \vec {e_i} \rangle \langle\vec x,\vec{e_i}\rangle \\
\langle \nabla_{\vec{e_i}}\vec v, \vec n \rangle \langle\vec x,\vec{e_i}\rangle \notag &= - \langle \nabla_{\vec n}\vec v, \langle \vec x,\vec {e_i}\rangle \vec{e_i}\rangle \\
\langle \nabla_{\vec{e_i}}\vec v, \vec n \rangle \langle\vec x,\vec{e_i}\rangle \notag &= - \langle \nabla_{\vec n}\vec v, \vec x -\langle\vec x, \vec n\rangle\vec n\rangle \\
\langle \nabla_{\vec{e_i}}\vec v, \vec n \rangle \langle\vec x,\vec{e_i}\rangle \notag &= - \langle \nabla_{\vec n}\vec v, \vec x\rangle +\langle\vec x,\vec n\rangle\langle \nabla_{\vec n}\vec v,\vec n\rangle\\
\langle \nabla_{\vec{e_i}}\vec v, \vec n \rangle \langle\vec x,\vec{e_i}\rangle \notag &= - \langle \nabla_{\vec n}\vec v, \vec x \rangle \\
\langle \nabla_{\vec{e_i}}\vec v, \vec n \rangle \langle\vec x,\vec{e_i}\rangle \notag &= \langle \nabla_{\vec x}\vec v, \vec n \rangle \\
\end{align}

We now switch to Euclidean coordinates.

\begin{align}\vec x \notag &= x\vec i +y\vec j +z\vec k \\
\vec v \notag &= -y\vec i+x\vec j \\
\vec n \notag &= n_1\vec i+n_2\vec j +n_3\vec k \\
\langle\vec v,\vec n\rangle \notag &= xn_2-yn_1 \\
\nabla_{\vec x}\vec v \notag &= x\nabla_{\vec i}\vec v +y\nabla_{\vec j}\vec v +z\nabla_{\vec k}\vec v \\
\nabla_{\vec x}\vec v \notag &= x\vec j - y\vec i \\
\langle \nabla_{\vec x}\vec v, \vec n \rangle \notag &= xn_2-yn_1 \\
\langle \nabla_{\vec x}\vec v, \vec n \rangle \notag &= \langle\vec v,\vec n\rangle
\end{align}

Thus $Lf=0$, so $f$ is either identically $0$ or it is an eigenvector of $L$ with corresponding eigenvalue $0$.  Recall from Definitions \ref{10} and \ref{16} that since $\Sigma_\infty$ is stable it must have no negative eigenvalues.  Thus the lowest possible eigenvalue is $0$.

Suppose by way of contradiction that $f$ is not identically $0$.  Then it must be an eigenfunction of $L$ with eigenvalue $0$, which would make $0$ the lowest eigenvalue of $L$.  However, Colding and Minicozzi showed in \cite{CM1} that the eigenfunction for the lowest eigenvalue of $L$ cannot change sign.  Thus we can assume (by possibly multiplying it by $-1$) that $f\geq 0$ on all of $\Sigma_\infty$.  This leads to a contradiction.

Consider the function $h:\Sigma_\infty\rightarrow \R$ defined by $h(x,y,z)=z$ for all $(x,y,z)\in \Sigma_\infty$.  Then by Sard's Theorem (Theorem \ref{31}) the set of critical points of this function maps to a set of measure $0$.  Then for a generic value $z=c$, the slice $\Sigma_\infty\cap\{z=c\}$ contains no critical points of $h$.  Therefore, for such a $c$, the slice $\Sigma_\infty\cap\{z=c\}$ consists of a disjoint union of curves.  Consider one of these curves, and call it $\gamma$.

Parameterize $\gamma$ by arclength $s$.  Then $\gamma=(x(s),y(s),c)$, and $\gamma'=(x'(s),y'(s),0)$.  Note that since there is no critical point of the function $h$ when $z=c$, at every point of $\gamma$ the tangent plane to $\Sigma_\infty$ does not equal the slice $\{z=c\}$.  Thus, at each point of $\gamma$ the projection of the unit normal to $\Sigma_\infty$ onto the slice $\{z=c\}$ is nonzero.  This projection, call it $\vec n ^{\top}$, is also perpendicular to $\gamma'$, since $\gamma'$ is tangent to $\Sigma_\infty$.  Thus we can assume that $\vec n ^{\top}=l(s)(-y',x',0)$, where $l(s)>0$.  Note that $\vec v=(-y,x,0)$, so $$f=\langle \vec v, \vec n\rangle=l(s)(xx'+yy').$$

However, $r^2=x^2+y^2,$ so $2rr'=2xx'+2yy'$ which implies $r'=\frac f{rl}$.  We know $r$ and $l$ are strictly positive, so $f$ is positive at a point on $\gamma$ if and only if moving along $\gamma$ in the direction of increasing $s$ causes $r$ to increase.  Likewise, $f$ is negative if and only if $r$ is decreasing along $\gamma$.

We know that $f\geq 0$.  Thus if $\gamma$ is a closed curve, $\gamma$ must be a circle centered on the $z$-axis.  If $\gamma$ is not closed, then consider moving along $\gamma$ in the direction of decreasing $s$.  Then the distance $r$ to the $z$-axis can never increase.  Also, $\Sigma_\infty$ is embedded, so it cannot intersect itself.  Since $\Sigma_\infty$ does not intersect the interior of the cylinder, $r$ can never decrease below $\sqrt 2$.  Thus, $\gamma$ must be a spiral.  However, this cannot happen since by construction $\Sigma_\infty$ minimizes area.  Thus, $\gamma$ must be a circle centered on the $z$-axis, and $f\equiv 0$ on a dense set and hence everywhere.  Likewise, $\Sigma_\infty$ is rotationally symmetric on a dense set of $z$ slices.  Since $\Sigma_\infty$ has no singularities, it must be rotationally symmetric everywhere.  This completes the proof.
\end{proof}

We now state a theorem by Kleene and M{\o}ller that gives us more information about the surface $\Sigma_\infty$ found in Theorem \ref{22}.

\begin{thm}\label{23}\cite{KM} For each fixed ray from the origin, $$r_{\sigma}(z)=\sigma z, \ \ \ r_\sigma :(0,\infty)\rightarrow \R^+, \ \ \ \sigma>0,$$ there exists a unique smooth graphical solution $u_\sigma :[0,\infty)\rightarrow \R^+$ which is asymptotic to $r_\sigma$ and whose rotation about the $z$-axis is a self-shrinker.

Also, for $d>0,$ any solution $u:(d,\infty)\rightarrow \R^+$ (whose rotation about the $z$-axis is a self-shrinker) is either the cylinder $u\equiv \sqrt 2$, or is one of the $u_\sigma$ for some $\sigma=\sigma(u)>0$.

Furthermore, the following properties hold for $u_\sigma$ when $\sigma>0$:
\begin{enumerate}
\begin{item}$u_\sigma>r_\sigma$, and $u(0)<\sqrt 2,$
\end{item}

\begin{item}$|u_\sigma(z)-\sigma z|=O(\frac 1 z ),$ and $|u'_\sigma(z)-\sigma|=O(\frac 1 {z^2})$ as $z\rightarrow \infty,$
\end{item}

\begin{item}$\Sigma_\sigma$ generated by $u_\sigma$ has mean curvature $H(\Sigma_\sigma)>0$,
\end{item}

\begin{item}$u_\sigma$ is strictly convex, and $0<u'_\sigma<\sigma$ holds on $[0,\infty)$,
\end{item}

\begin{item}$\gamma_\sigma$, the maximal geodesic containing the graph of $u_\sigma$, is not embedded.
\end{item}
\end{enumerate}
\end{thm}

The next corollary follows immediately from Theorem \ref{32} and the uniqueness in Theorem \ref{23}.

\begin{cor} For each $a\in(0,\sqrt 2)$, the stable surface $\Sigma_\infty$ found in Theorem \ref{22} must be a piece of one of the surfaces of rotation $\Sigma_\sigma$ described in Theorem \ref{23}.  In particular, there exists some $\sigma$ such that $\Sigma_\infty$ must equal the portion of $\Sigma_\sigma$ contained in the exterior of the cylinder of radius $\sqrt 2$.
\end{cor}

We now have a stable half-infinite self-shrinker with boundary equal to $\gamma_a$ for each $a>0$.  By symmetry, we have also solved the problem for all $a<0$.  We now extend this result to $a=0$.

\begin{thm}\label{49} The portion of the flat plane in $\R^3$ given by $\{|\vec x|>\sqrt 2\}$ is a stable self-shrinker.
\end{thm}
\begin{proof} We note first that for each $a\in(0,\sqrt 2)$, we obtain a stable self-shrinker given by one of the surfaces described in Theorem \ref{23}.  Each of these surfaces is contained between the flat plane $\{z=0\}$ and the cone obtained by rotating the ray $r_{\sigma}(z)=\sigma z$ about the $z$-axis.  Then, letting $\sigma\rightarrow\infty$ so that the cones approach the plane, we obtain a sequence of stable self-shrinkers converging pointwise to the portion of the plane given by $\{|\vec x|>\sqrt 2\}$.  In particular, this convergence is uniform on compact sets.

In the conformal metric $d\tilde\mu=e^{\frac{-|\vec x|^2} 4}d\mu$ these converging surfaces are area minimizing.  We now follow an argument from \cite{MY} to show that the limit surface is also area minimizing and hence stable as a self-shrinker in the standard metric.  Suppose the limit surface does not minimize area.  This means that there exists some compact subdomain $\Sigma$ of the limit surface such that some other surface $\tilde \Sigma$ satisfies $\partial\tilde\Sigma=\partial\Sigma$ and $Area(\Sigma)=Area(\tilde\Sigma)+\epsilon$ for some $\epsilon>0$.  However, there is a sequence of area minimizing compact surfaces with boundary converging uniformly to $\Sigma$.  The boundaries are also converging uniformly.

Thus, there exists some area minimizing $\Sigma_m$ such that $Area(\Sigma)\leq Area(\Sigma_m) +\frac{\epsilon}2$, and the boundaries of $\Sigma$ and $\Sigma_m$ are close enough together that they bound some surface $\tilde A$ with $Area(\tilde A)<\frac{\epsilon}2$.  However, this gives us a contradiction.  The surface $\tilde\Sigma\cup\tilde A$ has the same boundary as $\Sigma_m$.  On the other hand, we know $$Area(\tilde\Sigma\cup\tilde A)<Area(\Sigma)-\frac{\epsilon}2 \leq Area(\Sigma_m).$$  This contradicts the fact that $\Sigma_m$ minimizes area.  This completes the proof.
\end{proof}

We recall attention to the value $r_1\approx 2.514$ from Proposition \ref{48}.  We have now shown that for any value $r\in(\sqrt 2 , r_1)$, the circle of radius $r$ in the $2$-plane splits the plane into two stable regions.

\chapter{Future Research}\label{61}

We showed in Chapter \ref{24} that generalized cylinders split in a unique way into maximal stable rotationally symmetric regions.  However, we showed at the end of Chapter \ref{60} that this is not true for the plane in $\R^3$.  This result relies on a minimal surface compactness result in Riemannian $3$-manifolds (Theorem \ref{21}), so we have been unable to reproduce it in higher dimensions.  We know the index of the hyperplane in any dimension is $1$ (see Theorem \ref{37}), and we also know that the positive Jacobi field $f_1$ on the region $\{r>r_1\}$ is not an eigenfunction in any dimension (see Lemma \ref{62}).  These facts lead us to the following.

\begin{Con} Let $\Sigma=\R^n\subset\R^{n+1}$ be a self-shrinker with $n\geq 3$.  Let $r_1$ be as in Remark \ref{55}.  Note that $r_1$ depends on $n$.  Then there exists some $r_0=r_0(n)$ with $r_0<r_1$ such that for all $r\in(r_0,r_1)$ the $(n-1)$-sphere of radius $r$ centered at the origin splits $\Sigma$ into two stable regions.
\end{Con}

Recall that self-shrinkers are also minimal surfaces with respect to a Gaussian metric.  Minimal surfaces solve a variational problem, so there is also a closely related isoperimetric problem.  This problem is to minimize surface area within the class of surfaces which bound a fixed volume.  Note that the volume of all of $\R^n$ with the Gaussian metric is finite, so any surface that splits $\R^n$ into two connected components bounds a region of finite volume.

Some work has already been done on this subject, and in fact Carlen and Kerce have proven that solutions of this isoperimetric problem are flat hyperplanes \cite{CK}.  In Proposition \ref{39} we showed that all hyperplanes through the origin have stability index $1$.  It should be possible to extend this result to the hyperplanes that do not go through the origin and prove the following.

\begin{Con} Every hyperplane in $\R^n$ endowed with the Gaussian metric is stable as a solution to the isoperimetric problem.
\end{Con}

\begin{appendix}
\end{appendix}
%%%%%%%%%%%%%%%%%%%%%%%%%%%%%%%%%%%%%%%%%%%%%%%%%%%%%%%%%%%%%%%%%%%%%%%%%%%%%%%%%%%%%%%%%%%%%%%%%
%%%%%%%%%%%%%%%%%%%%%%%%%%%%%%%%%%%%%%%%%%%%%%%%%%%%%%%%%%%%%%%%%%%%%%%%%%%%%%%%%%%%%%%%%%%%%%%%%
\addcontentsline{toc}{chapter}{\protect\numberline{}Bibliography}

\pagebreak

\begin{center}Vitae\end{center}

Caleb Hussey was born on March 13th, 1982 in West Union, Ohio.  He received his Bachelor of Arts in Mathematics from New College of Florida in June 2005.  His undergraduate thesis was completed under the guidance of Dr. David T. Mullins.  He entered a Doctoral Program at The Johns Hopkins University in the fall of 2006.  He received his Master of Arts in Mathematics from The Johns Hopkins University in May 2007. His dissertation was completed under the guidance of Dr. William P. Minicozzi II, and this dissertation was defended on June 27th, 2012.

\addcontentsline{toc}{chapter}{\protect\numberline{}
Vitae}

\end{document}